\documentclass[final,3p,times]{elsarticle}
\usepackage{graphicx}
\usepackage{subcaption}
\usepackage[T1]{fontenc}
\usepackage{ dsfont }
\usepackage{amsthm}
\usepackage{amsmath}
\usepackage[dvipsnames]{xcolor}
\usepackage{mathrsfs}
\usepackage{amssymb}
\usepackage[T1]{fontenc}
\usepackage{ dsfont }
\usepackage{amsthm}
\usepackage{amsmath}
\newtheorem{lemma}{Lemma}
\usepackage{hyperref}
\theoremstyle{plain}
\newtheorem{theorem}{Theorem}[section]
\newtheorem{definition}{Definition}[section]
\theoremstyle{remark}
\newtheorem{remark}{Remark}
\newtheorem{corollary}{Corollary}
\newtheorem{assump}{}
\usepackage{natbib}

\newcommand{\lV}{\lVert}
\newcommand{\rV}{\rVert}
\newcommand{\lf}{\left}
\newcommand{\rg}{\right}

\usepackage{amssymb}
\usepackage{amsmath}
\usepackage{amsthm}

\usepackage[dvipsnames]{xcolor}
\usepackage{mathrsfs}
\usepackage{amssymb}
\usepackage[T1]{fontenc}
\usepackage{ dsfont }
\usepackage{amsthm}
\usepackage{amsmath}
\usepackage{hyperref}
\theoremstyle{plain}
\theoremstyle{remark}
\usepackage{natbib}

\begin{document}


\begin{frontmatter}

\title{Existence and uniqueness of solutions of unsteady Darcy-Brinkman problem for modelling miscible reactive flows in porous media} 
 \author{Pankaj Roy\fnref{label2}}
 \ead{pankaj.roy@iitg.ac.in}
 \author{Satyajit Pramanik\corref{cor1}\fnref{label2}}
\ead{satyajitp@iitg.ac.in}
\cortext[cor1]{Corresponding author.}

\affiliation[label2]{organization={Department of Mathematics, Indian Institute of Technology Guwahati},
            city={Guwahati},
            postcode={781039}, 
            state={Assam},
            country={India}}

\begin{abstract}
In this work, we investigate a model describing flow through porous media with permeability heterogeneity, combining an advection-reaction-diffusion equation for solute concentration with an unsteady Darcy-Brinkman equation with Korteweg stresses in the presence of external body forces for the flow field. Such models are appropriate in describing flows in fractured karst reservoirs, mineral wool, industrial foam, coastal mud, etc. These equations are coupled with Neumann boundary conditions for the solute concentration and no-flow conditions for the fluid velocity. For a broad class of initial data, we proved the existence of weak solutions. In the presence of a second-order nonlinear reaction, we show that the long-time behaviour of the solution depends on the initial concentration \(C_0\). More precisely, the solution exists for all time if \(0\leq C_0\leq 1\), and blows up at finite time if $C_0>1$. Furthermore, the uniqueness of the solution is proved for a two-dimensional domain. Finally, numerical simulations based on the finite element method have been presented that illustrate non-negativity of the concentration, long-time decay, and finite-time blow-up in agreement with theoretical estimates. 
\end{abstract}

\end{frontmatter}
 
\section{Introduction}

Porous media flows are ubiquitous in various configurations from biological to geophysical to engineering contexts. These processes are governed by the complex interplay of fluid properties and the structure of the porous medium. The mathematical description of flow through porous media typically involves a nonlinear partial differential equations system. Central to modeling these flows is Darcy's law, an empirical relationship given by $ \displaystyle \frac{\mu}{K}\mathbf u = -\boldsymbol{\boldsymbol{\nabla}} p$, proposed by Henry Darcy, which describes the flow of a fluid through a porous medium based on factors such as pressure gradients and permeability \citep{darcy1856}. Here, $\mathbf u, p, \mu, K$ denote the fluid velocity, pressure, viscosity of the fluid, and permeability of the porous media, respectively. Darcy's law has been widely used to understand and predict flow behaviour in various environmental, industrial, and biological systems, which are significantly influenced by the nature and structure of the porous medium. Notable examples include the study of carbon dioxide (CO$_2$) sequestration for mitigating climate change \cite{huppert2014fluid}, analyzing groundwater contamination to ensure safe water resources \citep{neuzil1986groundwater}, enhancing oil recovery from natural reservoirs to meet energy demands \cite{homsy1987viscous}, and investigating tumor growth \cite{zheng2022tumor}, where nutrient and drug transport through porous tissues is critical \citep{molz1986simulation}. These diverse applications highlight the universal importance of Darcy's law in understanding flow through porous media. 

Although used widely, Darcy's law comes with its limitations. \citet{khalifa2002new} showed it breaks down when the critical Reynolds number is exceeded in steady flows or the critical acceleration is reached in unsteady flows. For such unsteady flow situations, an extension to Darcy's law has been presented by adding a time derivative of velocity to the classical Darcy's law. Again, for porous media with large porosity (typically, greater than 0.75), Darcy's law fails to be an appropriate model \citep{mccurdy2019convection}. 
Instead, the Brinkman equation better represents viscous flows in such media, enriched with micro and macro length scales \citep{brinkman1949calculation}. This extension to Darcy's law includes the Laplacian of the velocity vector and is given by $ \displaystyle \frac{\mu}{K}\mathbf u = - \boldsymbol{\boldsymbol{\nabla}} p + \mu_e \Delta \mathbf u$, where $\mu_e$ is the effective viscosity and other symbols have the same meaning as mentioned above. The effective viscosity typically depends on the structure of the porous media and the strength of the flow field, leading to values different from the dynamic viscosity of the fluid \citep{vafai2015handbook}. The Brinkman equation has been widely used to model fractures in porous media \citep{morales2017darcy}, tumour growth \citep{ebenbeck2019cahn}, mixing of surface and ground water \cite{layton2002coupling}, cancer cell migration, platelet aggregation \citep{link2020mathematical}, flow in porous conduits \citep{babuvska2010residual}, to name a few. Inertial effects in porous media flows -- modelled using the Darcy-Frochheimer equation \citep{cimolin2013navier} -- have been of interest to researchers over several decades \citep[see][for a review]{koch2001inertial}. Recently, inertial effects have been revisited by researchers with a new outlook to explore the fundamental physical mechanisms of transition from Darcy regime to inertial (Darcy-Forchheimer) regime \citep[see][and refs. therein]{naqvi2025inertial, strzelczyk2025role}. Although the inclusion of time derivative in Darcy/Darcy-Brinkman equations can be debated, researchers have considered an unsteady Darcy-Brinkman equation in modeling biological processes, petroleum extraction in karst reservoirs, etc. \citep[see][and references therein]{kundu2024existence}. \citet{CELEBI2006801} showed the well-posedness of a Brinkman–Forchheimer model. \citet{Girelli_unsteady_DB} studied a multiscale dual porosity model, where the fluid flow within the blood vessel is governed by Darcy's law, and the flow within the matrix region, corresponding to the conduit network formed by fibroblastic reticular cells, is governed by the unsteady Darcy-Brinkman equation, using an asymptotic homogenization technique. Well-posedness of an unsteady Darcy-Brinkman model with a coupling of the advection-diffusion equation is established in \cite{kundu2024existence}. The unsteady Darcy-Brinkman equations have also gained the attention of numerical analysts \citep{hou2016dual, evje2018stokes}.

Fluid flows in porous media are often connected with heat and mass transfer \cite{carbonell1984heat, delgado2011heat, homsy1987viscous, pramanik2013linear, dehghan2024fully, caucao2020fully, caucao2022posteriori}. In 1901, \citet{korteweg1901form} introduced the idea that nonuniform distributions of density, concentration, or temperature within a fluid can generate stresses and induce convection. The effects of these stresses (or an effective tension), known as the Korteweg stresses, have been widely investigated experimentally and numerically. \citet{pojman2007miscible} experimentally studied that there exists a transient interfacial phenomenon similar to those observed with immiscible fluids for both the isothermal and thermal gradient miscible fluids (e.g., water-honey). Also, various authors have studied the effect of Korteweg stress in the context of miscible viscous fingering \citep{swernath2010effect, chen2001miscible, chen2002miscible}.

Miscible flows in porous media are commonly modeled by coupling a mass balance equation of a scalar species of advection-diffusion-reaction type with the flow equations \citep{homsy1987viscous, pramanik2013linear, de2016chemo, rana2019influence}. Miscible displacements in porous media have been a focus of research by researchers, of which fingering instabilities due to mobility and density differences are of special interest in the contexts of enhanced oil recovery \citep{homsy1987viscous}, CO$_2$ sequestration \citep{huppert2014fluid}, chromatography separation \citep{rana2019influence}, etc. The underlying fluids are sometimes reactive. Autocatalysis is a chemical process where the product acts as a catalyst, accelerating its formation. These kinds of reactions start very slowly but rapidly increase in rate as more products begin to form. A typical form is 
\(
A + nB \rightarrow (n + 1)B,
\)
where \(n\) is the autocatalysis order, and \(A\) and \(B\) interact to produce \(B\), where \( B \) catalyzes its own production. \citet{de2001fingering} showed that an auto-catalytic reactive front leads to new instability patterns in miscible porous media flows. \citet{nagatsu2014hydrodynamic} experimentally demonstrated that a precipitation reaction alters the permeability of the Hele-Shaw cell, inducing fingering instabilities similar to viscous fingering caused by viscosity variation with the chemical reactants/product. \citet{Ghesmat2009} investigated second and third-order autocatalytic reactions and reported the major differences in the flow development. In the present study, a second-order auto-catalytic reaction of the form $-\kappa C(1-C)$ is considered.  
In the present study, a second-order reaction kinetics of the quadratic type $- \kappa c(1-c)$, $\kappa > 0$ has been considered. This nonlinearity is often used as a prototype for autocatalytic reactions \citep{fisher1937wave} and has physical significance in combustion processes \citep{williams2018combustion}. However, such reactions introduce substantial analytical challenges, as this feature $c = 1$ as an unstable state for $\kappa > 0$ in a spatially homogeneous system.

Although miscible displacements in porous media have gained the attention of physicists, engineers, chemists, geophysicists, and numerical analysts, the well-posedness of a certain class of problems describing miscible porous media flow is poorly explored \citep[see][for some recent advances in this direction]{migorski2019nonmonotone, allali2017existence, kostin2003modelling}. These researchers primarily focused on the cases when either the viscosity of the fluid or the permeability of the porous media varies with the concentration of a scalar species. Their studies are also limited to the case of a first-order reaction. In this paper, we studied the existence and uniqueness of a weak solution of reactive miscible flows in porous media. We considered the viscosity and permeability to be positive functions of the concentration of a solute, which undergoes a second-order chemical reaction. 

\subsection{Mathematical Model} \label{subsec: Mathematical Model}

In this section, we present the unsteady Darcy-Brinkman model, which characterizes the flow of a miscible fluid having variable viscosity within a porous medium featuring heterogeneous permeability, alongside the inclusion of the Korteweg stress. 

The fluid we considered here is incompressible and miscible. Let $\Omega \subset \mathds{R}^N$ $(N = 2,3)$ be the bounded domain with $\mathcal C ^1$ boundary of the flow and $\left( 0, T \right)$ be the time interval. The governing equations are \citep{kundu2024existence, Ghesmat2009, fisher1937wave, williams2018combustion}, 
\begin{eqnarray}
	& & \boldsymbol{\nabla} \cdot \mathbf{u} =0  \quad  in\  \Omega \times \left(0,T\right) \label{eq:continuity}, \\
	& & \frac{\partial \mathbf{u}}{\partial t}+\frac{\mu\left(C\right)}{K\left(C\right)} \mathbf{u}=-\boldsymbol{\nabla} p+\mu_e\Delta \label{eq:Brinkman-Darcy} \mathbf{u}+\boldsymbol{\nabla} \cdot \mathbf{T}\left(C\right)+\mathbf f \quad  in\  \Omega \times \left(0,T\right), \\
	& & \frac{\partial C}{\partial t}+\mathbf{u} \cdot \boldsymbol{\nabla} C =d \Delta C  -\kappa C\left(1-C\right)\quad  in\  \Omega \times \left(0,T\right) \label{eq:transport}.
\end{eqnarray}
Here, \( C \) denotes the solute concentration in the fluid, and \( d \) represents its diffusion coefficient. The parameter \( \kappa \) is the reaction constant, \( \mathbf f \) denotes an external force, and $\mathbf T(C)$ denotes the Korteweg stress tensor. All other symbols retain the meanings defined earlier.
The viscosity \( \mu(C) \) of the fluid and the permeability \( K(C) \) of the porous medium are both assumed to depend on the solute concentration. Furthermore, all physical parameters—namely, the viscosity \( \mu(C) \), permeability \( K(C) \), effective viscosity \( \mu_e \), diffusion coefficient \( d \), and reaction constant $\kappa$—are taken to be strictly positive. 

The motivation for considering viscosity and permeability as functions of the solute concentration \( C \) arises from applications such as enhanced oil recovery \cite{homsy1987viscous}, studies on viscous fingering \cite{nagatsu2014hydrodynamic, rana2019influence}, etc.
 For instance, in \cite{nagatsu2014hydrodynamic}, the permeability was modeled as \( K(C) = \exp(-2RC) \), where \( R \) is a constant parameter. We have denoted the ratio of viscosity and permeability as the mobility function $F(C)$, 
 \begin{equation}
     \label{eq:F}
     F\left(C\right)= \displaystyle \frac{\mu \left(C\right)}{K\left(C\right)}.
 \end{equation}
The Korteweg stress tensor \( \mathbf{T}(C) \) is expressed as \citet[see][and references therein]{pramanik2013linear}
\begin{align}\label{exp: Korteweg 1st}
\mathbf{T}(C) = \left( -\frac{1}{3} \hat{\delta} |\boldsymbol{\nabla} C|^{2} + \frac{2}{3} \gamma \boldsymbol{\nabla}^{2} C \right) \mathbf I + \hat{\delta} \boldsymbol{\nabla} C \boldsymbol{\nabla} C,
\end{align}
where \( \mathbf I \) is the identity matrix. The constants \( \gamma > 0 \) and \( \hat{\delta} > 0 \) are known as the Korteweg parameters. From the continuum theory of surface tension \cite{JOSEPH1996104}, one can write,
\begin{align} \label{exp: Korteweg stress}
\boldsymbol{\nabla} \cdot \mathbf{T}(C) = \boldsymbol{\nabla} Q(C) - \hat{\delta} \boldsymbol{\nabla} \cdot \left( \boldsymbol{\nabla} C  \boldsymbol{\nabla} C \right),
\end{align}
where \( Q(C) \) is the term appearing as the coefficient of the identity matrix \( \mathbf I \) in \( \mathbf{T}(C) \).
 
The governing equations \eqref{eq:continuity}--\eqref{eq:transport} are subject to the following boundary conditions:
\begin{equation}
	\mathbf{u} = \mathbf 0, \quad \frac{\partial C}{\partial \boldsymbol{\eta}} = 0 \quad on \ \partial \Omega \times \left(0,T\right) \label{eq: boundary condition},
\end{equation}
where $\boldsymbol{\eta}$ represents the unit outward normal vector to the boundary \(\partial \Omega\).
The first condition is due to the no-slip boundary condition, and the second condition is due to the impermeability of the boundary.
Additionally, the initial conditions for the equations \eqref{eq:continuity}--\eqref{eq:transport} are given as:
\begin{equation}
	\mathbf{u}\left(\mathbf{x}, 0\right) = \mathbf{u}_0\left(\mathbf{x}\right), \quad C\left(\mathbf{x}, 0\right) = C_0\left(\mathbf{x}\right) \quad \forall \mathbf{x} \in \Omega \label{eq:initial_condition}.
\end{equation}

The remainder of the paper is organized as follows. We introduce necessary function spaces and the weak formulation of the problem \eqref{eq:continuity} -- \eqref{eq:transport} in Section \ref{sec:function_spaces}. Section \ref{sec: main_result} develops necessary lemmas and establishes the proof of the main result --- existence and uniqueness of the weak solutions and long-time behaviour of the solute concentration. Section \eqref{sec:special_cases} discusses some special cases of the mobility function $F(C)$.
Finally, numerical results demonstrating behaviour of solute concentration in agreement with theoretical estimates are presented in Section \eqref{sec:numerical}, followed by the concluding remarks in Section \ref{sec: conclusion}. 

\section{Function spaces and preliminary results} \label{sec:function_spaces}

In this section, we introduce standard notions of necessary function spaces and essential results that will be frequently referred to throughout this paper. 
For a Hilbert space \( \mathscr{H} \), the inner product on \( \mathscr{H} \) is represented by \( (\cdot, \cdot)_{\mathscr{H}} \), and the associated norm is given by \( \|\cdot\|_{\mathscr{H}} \). We denote the duality pairing between \( \mathscr{H}\) and its dual \( \mathscr{H}^* \) by \( \langle \cdot, \cdot \rangle_{\mathscr{H}^*,\mathscr{H}} \).  For a Banach space \( \mathscr{B} \) we denote the associated norm by \( \|\cdot\|_{\mathscr{B}} \). The space \( (L^2(\Omega))^2 \) consists of vector-valued functions \( \mathbf{u} = (u_1, u_2) \), where each component \( u_i \) (for \( i = 1,2 \)) belongs to the space \( L^2(\Omega) \). This space is equipped with the inner product,  
\begin{align}
    (\mathbf{u}, \mathbf{v})_{(L^2(\Omega))^2} = \int_{\Omega} \left( u_1 v_1 + u_2 v_2 \right) \, dx,
\end{align}
and the corresponding norm, 
\begin{equation}
    \label{eq:norm}
    \|\mathbf{u}\|_{(L^2(\Omega))^2} = \left( \|u_1\|_{L^2(\Omega)}^2 + \|u_2\|_{L^2(\Omega)}^2 \right)^{\frac{1}{2}}. 
\end{equation}
For notational convenience, we have used $\|\cdot\|_{(L^2(\Omega))^2}$, $\|\cdot\|_{L^2(\Omega)}$ by $\|\cdot\|_{L^2}$ and $(\cdot,\cdot)_{(L^2(\Omega))^2}$, $(\cdot,\cdot)_{L^2(\Omega)}$ by $(\cdot,\cdot)$.
For the space $H^1_0$ and its dual $H^{-1}$ we have denoted the duality pairing between them by \[
\langle f, v \rangle, \quad \text{for } f \in H^{-1}(\Omega) \text{ and } v \in H^1_0(\Omega).
\]
Now, we present some key preliminary results required for the development of our main arguments.

\begin{theorem}[Gagliardo–Nirenberg (see \emph{\citep[Lemma 1]{migorski2019nonmonotone}}, \emph{\citep[Lemma 1.1]{ebenbeck2019cahn}})]\label{Th: gagliardo}
Let $\Omega \subset \mathds R^N,\ N = 2,3, $ be a domain with $\mathcal{C}^{1}$ boundary. Then, for any $\boldsymbol{\phi} \in H^{1}(\Omega)$, there exists a constant  $M>0 $ depending on $\Omega$ such that the following inequality holds for $N=2$:
$$
\|\boldsymbol{\phi}\|_{L^{4}} \leq M\|\boldsymbol{\phi}\|_{L^{2}}^{1 / 2}\| \boldsymbol{\phi}\|_{H^{1}}^{1 / 2},
$$
and for $N=3$:
$$
\|\boldsymbol{\phi}\|_{L^{3}} \leq M\|\boldsymbol{\phi}\|_{L^{2}}^{1 / 2}\| \boldsymbol{\phi}\|_{H^{1}}^{1 / 2}, \quad \text{and} \quad \|\boldsymbol{\phi}\|_{L^{4}} \leq M\|\boldsymbol{\phi}\|_{L^{2}}^{1 / 4}\| \boldsymbol{\phi}\|_{H^{1}}^{3 / 4}.
$$
\end{theorem}

\begin{theorem}[{Gr\"onwall’s inequality (integral form)} (see Appendix of \citep{evans2010partial})]\label{Th: gronwall}
Let $\xi(t)$ be a nonnegative, summable function on $[0,T]$ which satisfies for a.e.\ $t$ the integral inequality
\begin{equation*}
\xi(t) \leq A_1 \int_0^t \xi(s)\, ds + A_2
\end{equation*}
for constants $A_1, A_2 \geq 0$. Then
\begin{equation*}
\xi(t) \leq A_2 (1 + A_1 t e^{A_1 t})
\end{equation*}
for a.e.\ $0 \leq t \leq T$.
\end{theorem}

\subsection{Weak formulation\label{subsec:Weak_formulation}}

In this section, we introduce the concept of a weak solution. To this end, we first define the necessary function spaces that will serve as the framework for our analysis,
\begin{eqnarray} 
    & &\mathscr S_1 = \left\{ \mathbf{v} \in \left(L^2(\Omega)\right)^2 : \boldsymbol{\nabla} \cdot \mathbf{v} = 0, \ \left. \mathbf{v} \cdot \boldsymbol{\eta} \right\rvert_{\partial \Omega} = 0 \right\}, \nonumber \\ 
    & & \mathscr V_1 = \left\{\mathbf{v} \in \left(H_0^1\left(\Omega\right)\right)^2:  \boldsymbol{\nabla}\cdot \mathbf{v} = 0 \right\}, \nonumber \\ 
    & &\mathscr S_2=H^1\left(\Omega\right), \nonumber \\ 
    & & \mathscr V_2 = \left\{B \in H^2\left(\Omega\right): \left. \frac{\partial B}{\partial \boldsymbol{\eta}}\right\rvert_{\partial \Omega} = 0 \right\}.\nonumber
\end{eqnarray}

\begin{definition}[Weak solution\label{def: weak solution}]
    A pair $\left(\mathbf u, C\right)$ is said to be a weak solution to the problem \eqref{eq:continuity}--\eqref{eq:initial_condition}, if it satisfies the following:
    \begin{enumerate}
    \item $\left(\mathbf{u},C\right)\in\left(L^2\left(0,T;\mathscr V_1\right),L^2\left(0,T;\mathscr V_2\right)\right)$.
    \item $\mathbf u\left(\cdot,0\right)=\mathbf u_0 $ and $C\left(\cdot,0\right)=C_0$ $a.e.$ in $\Omega$.
\item For every $\mathbf v \in \mathscr V_1$ and for every $B\in \mathscr V_2$,
        \begin{align}
   \left \langle  \frac{\partial \mathbf{u}(t)}{\partial t},\mathbf{v} \right \rangle + \mu_e\left(\boldsymbol{\nabla} \mathbf{u}\left(t\right), \boldsymbol{\nabla} \mathbf{v}\right)  +  \left(F\left(C\left(t\right)\right)\mathbf{u}\left(t\right),\mathbf{v}\right) -\left\langle \boldsymbol{\nabla} \cdot \mathbf{T}\left(C\left(t\right)\right) , \mathbf{v} \right\rangle 
    - \left(\mathbf f\left(t\right),\mathbf v\right) = 0  \quad a.e. \ on \ \left(0,T\right), \label{var u} \\
   \left.
   \begin{aligned}
  \left(\frac{\partial C\left(t\right)}{\partial t},B\right) + d\left(\boldsymbol{\nabla} C\left(t\right), \boldsymbol{\nabla} B\right)
   + \left(\mathbf{u}\left(t\right) \cdot \boldsymbol{\nabla} C\left(t\right),B\right) 
   + \kappa \left(C\left(t\right)\left(1-C\left(t\right)\right),B\right)=0 \quad a.e. \ on \ \left(0,T\right)  \quad \text{for}\ N =2,\\
   \left\langle\frac{\partial C\left(t\right)}{\partial t},B\right\rangle + d\left(\boldsymbol{\nabla} C\left(t\right), \boldsymbol{\nabla} B\right)
   - \left(\mathbf{u}\left(t\right) C\left(t\right), \boldsymbol{\nabla} B\right)   
   + \kappa \left(C\left(t\right)\left(1-C\left(t\right)\right),B\right)=0 \quad a.e. \ on \ \left(0,T\right) \quad \text{for}\ N =3.
  \end{aligned}\right
  \}\label{var C}
  \end{align}
  \end{enumerate}
\end{definition}

\begin{remark}
 Throughout the subsequent analysis, the variational formulation corresponding to the case of $N=2$ is used for notational simplicity. Only difference arises while proving Lemma \ref{lemma:L3}, where two different cases are shown.
\end{remark}

\subsection{Assumptions}\label{subsec:assumptions}
\begin{assump} \label{assump:A1}
The function \( F \) is assumed to be a mapping \( F: L^2(0,T;\mathscr V_2) \to L^2(0,T;H^1(\Omega)) \) and satisfies a Lipschitz-like condition:
\begin{equation}
    \|F(C_1) - F(C_2)\|^2_{L^2(0,T;H^1)} \le M \|C_1 - C_2\|^2_{L^2(0,T;L^2)}.
\end{equation}
\end{assump}

\begin{assump} \label{assump:A2}
The external body force $\mathbf f\left(t,\mathbf x\right)$ is taken to belong to $L^2\left(0,T;\left(L^2\left(\Omega\right)\right)^2\right)$.
\end{assump}

\begin{assump} \label{assump:A3}
To guarantee the existence and uniqueness of the pressure, we impose the condition $\int_\Omega p\, dx = 0$.
\end{assump}
\noindent

Unless mentioned otherwise, in the remainder of the paper, $\epsilon$ is considered to be a small positive constant, $M$ is a fixed positive constant, and $A\left(\epsilon\right)$ is a positive constant depending on $\epsilon$.

\section{Main results \label{sec: main_result}}

\begin{theorem}\label{th: existence}
Let $\Omega \subset \mathds R^N$ (N=2,3) and suppose that assumptions~\eqref{assump:A1}--\eqref{assump:A3} are satisfied. Then for any initial data \((\mathbf{u}_0, C_0) \in\mathscr S_1 \times  \mathscr S_2\), there exists a time $T>0$ such that the coupled system \eqref{eq:continuity}--\eqref{eq:initial_condition} admits a weak solution \((\mathbf{u}, C)\), in the sense of Definition~\ref{def: weak solution}. Moreover, the solution satisfies 
\begin{align*}
    (\mathbf u, C)\in \left(L^2(0,T;\mathscr V_1)\cap \mathcal C([0,T]; \mathscr S_1)\right)\times\left(L^2(0,T,\mathscr V_2)\cap \mathcal C([0,T]; \mathscr S_2)\right), \\ 
    \text{and} \quad\left(\frac{\partial \mathbf u}{\partial t},\frac{\partial C}{\partial t}\right) \in\left(L^{4/N}\left(0,T;\mathscr V_1^*\right),L^{2}\left(0,T;\mathscr V_2^*\right)\right).
\end{align*}
\end{theorem}
\begin{theorem} \label{th: bddness & blow_up}
    Under the assumptions of Theorem \ref{th: existence}, let $(\mathbf u, C)$ be a weak solution of the system \eqref{eq:continuity}--\eqref{eq:initial_condition}.
    \begin{enumerate}
        \item If \ $0\leq C_0\leq 1$ almost everywhere in $\Omega$, then $0\leq C \leq 1$ almost everywhere in $\Omega \times (0,T)$ and the solution exists for any $T>0$. Moreover, if $0\leq C_0\leq \gamma_1<1$ almost everywhere in $\Omega$, then the solute concentration converges exponentially to $0$.
        \item If $C_0\geq \gamma_2>1$ almost everywhere in $\Omega$, then $ C \geq \gamma_2$ almost everywhere in $\Omega \times (0,T)$ and the solute concentration $C$ blows up in a finite time.
    \end{enumerate}
\end{theorem}
\begin{theorem} \label{th: Uniqueness}
 Let $\Omega \subset \mathds R^2$ and suppose that assumptions~\eqref{assump:A1}--\eqref{assump:A3} are satisfied. Then for any initial data \((\mathbf{u}_0, C_0) \in\mathscr S_1 \times  \mathscr S_2\), the weak solution is unique.
\end{theorem}

\subsection{Proof of Theorem \ref{th: existence}}

\subsubsection{A Priori Estimates \label{sec: A Priori Estimates}}

The proof of Theorem \ref{th: existence} is based on the well-known Faedo-Galerkin method, as detailed in the works of \citet{Lions1969}, \citet{Temam1979, Temam1983}, and \citet{Ladyzhenskaya1963}. We first construct approximate solutions using the semi-discrete Galerkin method. We consider $\{\mathbf w_k\}_{k\in \mathds N}$ and $\{z_k\}_{k\in \mathds N}$ to be basis of $\mathscr V_1$ and $\mathscr V_2$, respectively, where $z_k$'s are eigen vectors of $``-\Delta"$ associated with the Neumann boundary condition and $\lambda_j$'s are corresponding eigen values.

Let $\left(\mathscr V_1\right)_n = span\left(\mathbf{w}_1, \mathbf w_2, \dots, \mathbf w_n\right)$ and $\left(\mathscr V_2\right)_n = span\left(z_1, z_2, \dots, z_n \right)$. We define $\mathbf{u}_n:[0,T] \to \ \left(\mathscr V_1\right)_n$ and $C_n: [0,T] \to \ \left(\mathscr V_2\right)_n$ by 
\begin{center}
	$\mathbf{u}_n\left(t\right)=\sum_{j=1}^{n}\alpha_j^n\left(t\right)\mathbf w_j$ and $C_n\left(t\right)=\sum_{j=1}^{n}\beta_j^n\left(t\right)z_j$.
\end{center}
and consider the following system of differential equations:
\begin{align}
	\left(\frac{\partial C_n\left(t\right)}{\partial t},z_j\right) & + d\left(\boldsymbol{\nabla} C_n\left(t\right), \boldsymbol{\nabla} z_j\right)+\left(\mathbf{u}_n\left(t\right) \cdot \boldsymbol{\nabla} C_n\left(t\right),z_j\right)\nonumber \\ 
    & +\kappa \left(C_n\left(t\right)\left(1-C_n\left(t\right)\right),z_j\right) = 0 \quad \forall z_j \in \left(\mathscr V_2\right)_n \quad a.e. \ on \ \left(0,T\right), \label{eq: fdC} \\
	\left \langle  \frac{\partial \mathbf{u}_n\left(t\right)}{\partial t},\mathbf w_j \right \rangle & + \mu_e\left(\boldsymbol{\nabla} \mathbf{u}_n\left(t\right), \boldsymbol{\nabla} \mathbf w_j\right) + \left( F\left(C_n\left(t\right)\right)
 \mathbf{u}_n\left(t\right),\mathbf w_j\right) \nonumber \\ 
 & -\left \langle \boldsymbol{\nabla} \cdot \mathbf{T}\left(C_n\left(t\right)\right), \mathbf w_j \right \rangle - \left(\mathbf f\left(t\right),\mathbf w_j\right) = 0 \quad \forall \mathbf w_j \in \left(\mathscr V_1\right)_n \quad a.e. \ on \ \left(0,T\right)  \label{eq: fdu},
 \end{align} 
with initial conditions, 
\begin{equation}
    \mathbf{u}_n\left(0\right)=\sum_{j=1}^{n}\left(\mathbf{u}_0,\mathbf w_j\right)\mathbf w_j
\end{equation}
and 
\begin{equation}
    C_n\left(0\right)=\sum_{j=1}^{n}\left(C_0,{z_j}\right)z_j. 
\end{equation}
Equation \eqref{eq: fdC}--\eqref{eq: fdu} correspond to a nonlinear system of
ﬁrst-order ordinary differential equations of the following form,
\begin{align}
    \frac{d \vec{\mathfrak{V}}(t)}{dt}= \vec{\mathfrak{F}} \left(\alpha_1^n(t),\alpha_2^n(t),\dots,\alpha_n^n(t),\beta_1^n(t),\beta_2^n(t),\dots,\beta_n^n(t)\right),
\end{align}
where $\vec{\mathfrak{V}}(t)=(\alpha_1^n(t),\alpha_2^n(t),\dots,\alpha_n^n(t),\beta_1^n(t),\beta_2^n(t),\dots,\beta_n^n(t))$. From equation \eqref{exp: Korteweg stress}, we obtain the following identity for any \( \mathbf v \in \mathscr V_1 \), utilizing the divergence free property of $\mathbf v$,
\begin{align}  
    \left\langle \boldsymbol{\nabla} \cdot T(B(t)), \mathbf{v} \right\rangle = -\hat{\delta} \left( \boldsymbol{\nabla} B(t) \Delta B(t), \mathbf{v} \right)  \quad a.e.\ on \ (0,T)
    \label{eq: Korteweg expression},
\end{align}  
where $B\in L^2(0,T;\mathscr V_2).$
Consequently, every term in \( \vec{\mathfrak{F}}(t) \), except for \( \left( F(C_n(t)) \mathbf{u}_n(t), \mathbf{w}_j \right) \), can be expressed in the form  
\[
\mathscr{K}\prod_{i=1}^{k_1} \prod_{j=1}^{k_2} \alpha_{a_i}^n(t) \beta_{b_j}^n(t),
\]  
where $\mathscr{K}$ is a constant, and \( a_i, b_j \) are certain indices in \( \{1,2,\dots,n\} \).
Utilizing this along with assumption \eqref{assump:A1}, we invoke Carathéodory’s existence theorem to establish the existence of approximate solutions $\mathbf u_n$ and $C_n$ on a time interval $(0,T_m]$. The following lemmas provide uniform boundedness of $\mathbf u_n$ and $C_n$. 
\begin{lemma}\label{lemma:L1}
$C_n$ is uniformly bounded in $L^\infty\left(0,T;L^2\left(\Omega\right)\right)\cap L^2\left(0,T;H^1(\Omega)\right).$
\end{lemma}

\begin{proof}
Let us consider the set
\begin{align*}
    \mathscr A_\epsilon = \left\{C_n: \int_\Omega C_n^2(\mathbf x,t)\ d\mathbf x  \le \epsilon \quad  a.e. \ on \ (0,T)\right\}.
\end{align*}
The construction of $\mathscr A_\epsilon$ gives us \begin{align}
    \|C_n\|_{L^\infty(0,T;L^2(\Omega))}<\infty \quad \forall C_n \in \mathscr A_\epsilon.
\end{align}
We now show the uniform boundedness of $C_n$ in $L^\infty(0,T;L^2(\Omega))$ for all $C_n\in \mathscr A_\epsilon^c.$
	By multiplying equation \eqref{eq: fdC} with $\beta_j^n$ and summing the resulting equations over $j=1$ to $ n$, and performing integration by parts, we obtain, 
	\begin{align}
    \frac{1}{2}\frac{d}{dt}\|C_n\left(t\right)\|_{L^2}^2+d\|\boldsymbol{\nabla}{C_n\left(t\right)}\|^2_{L^2}+\kappa \|C_n\left(t\right)\|_{L^2}^2 & =\kappa \left(C_n^2\left(t\right),C_n\left(t\right)\right) \nonumber \\ 
    & \le \kappa\|C_n(t)\|_{L^4}^2\|C_n(t)\|_{L^2} \quad a.e.\  on \ \left(0,T\right). \label{uniform boundedness of C_n 1st}
	\end{align}
Applying Gagliardo-Nirenberg (Theorem \ref{Th: gagliardo}) and Young's inequalities on the last term, we obtain
\begin{align}
    \frac{1}{2}\frac{d}{dt}\|C_n\left(t\right)\|_{L^2}^2+(d-\epsilon)\|\boldsymbol{\nabla}{C_n\left(t\right)}\|^2_{L^2}+(\kappa - \epsilon ) \|C_n\left(t\right)\|_{L^2}^2
    \le A(\epsilon) \|C_n\left(t\right)\|_{L^2}^{2N}\quad a.e.\  on \ \left(0,T\right). \label{uniform boundedness of C_n 2nd}
\end{align}
We choose $\epsilon$ such that the second and third terms of left hand side are non-negative. So, from inequality \eqref{uniform boundedness of C_n 2nd} we get,
\begin{align}
    \frac{d}{dt}\|C_n\left(t\right)\|_{L^2}^2\le  A(\epsilon)\|C_n\left(t\right)\|_{L^2}^{2N}\quad a.e.\  on \ \left(0,T\right). \label{uniform boundedness of C_n 3rd}
\end{align}
If \(\|C_0\|_{L^2} = 0\), then $C_n(0) = 0$ for all $n$. Then from inequality \eqref{uniform boundedness of C_n 3rd} we have $C_n=0$ for all \(n\in \mathds N,\) which eventually gives $C=0$ almost everywhere in \(\Omega\times(0,T)\). 

If \(\|C_0\|_{L^2} \ne 0\), then since $C_n(0)$ converges to $C_0$ in $L^2$ norm so there exist a positive integer $k$ such that $\|C_n(0)\|_{L^2(\Omega)}>0$ for all $n \ge k$. 
From inequality \eqref{uniform boundedness of C_n 3rd}, integrating from $0$ to $\tau$, for some $\tau \in (0,T)$, we get, for all $n \ge k$,
\begin{align}
    -\frac{1}{\|C_n\left(\tau\right)\|_{L^2}^{2N-2}} \le  -\frac{1}{\|C_n\left(0\right)\|_{L^2}^{2N-2}}  + A(\epsilon)\tau \le  -\frac{1}{\|C_0\|_{L^2}^{2N-2}} + A(\epsilon)\tau \quad \forall C_n \in \mathscr A_\epsilon^c.
\end{align}
This implies, there exists a $T>0$ such that $C_n$ is uniformly bounded in $L^\infty \left(0, T; L^2(\Omega)\right)$. Utilizing the uniform boundedness of $\|C_n\|_{L^\infty \left(0, T; L^2(\Omega)\right)}$, 
from inequality \eqref{uniform boundedness of C_n 2nd} we can say that $C_n$ is also uniformly bounded in $L^2\left(0,T;H^1\left(\Omega\right)\right)$.
\end{proof}
\begin{remark}\label{rem: bound of C(1-C)}
Using Lemma \ref{lemma:L1}, it can be shown that $C_n\left(1-C_n\right)$ is uniformly bounded in $L^2\left(0,T; L^2(\Omega)\right)$ space.
\end{remark}
\begin{lemma}
	$\mathbf{u}_n$ is uniformly bounded in $L^\infty\left(0,T; \mathscr S_1\right)\cap L^2\left(0,T;\mathscr V_1\right)$ and $C_n$ is uniformly bounded in $L^\infty\left(0,T; \mathscr S_2\right)\cap L^2\left(0,T;\mathscr V_2\right)$.\label{lemma:L2}
\end{lemma}
\begin{proof}Multiplying equation \eqref{eq: fdC} by $\hat \delta \lambda_j \beta_j^n$ and summing over $j=1$ to $n$,  we obtain the following equation,
\begin{align}
\frac{\hat{\delta}}{2} \frac{d}{dt} \left\| \boldsymbol{\nabla} C_n\left(t\right) \right\|_{L^2}^2 & + d\hat{\delta} \left\| \Delta C_n\left(t\right) \right\|_{L^2}^2
= \hat{\delta} \left( \mathbf{u}_n\left(t\right) \cdot \boldsymbol{\nabla} C_n\left(t\right), \Delta C_n\left(t\right) \right) \nonumber \\ & + \hat \delta \kappa\left(C_n\left(t\right)\left(1-C_n\left(t\right)\right),\Delta C_n\left(t\right)\right) \quad a.e.\  on \ \left(0,T\right). \label{eq: L2C}
\end{align}
Applying Cauchy-Schwarz and Young's inequality on the last term of \eqref{eq: L2C}, we get the following inequality,
\begin{align}
\frac{\hat{\delta}}{2} \frac{d}{dt} \left\| \boldsymbol{\nabla} C_n\left(t\right) \right\|_{L^2}^2 + \left(d\hat{\delta}-\epsilon\right) \left\| \Delta C_n\left(t\right) \right\|_{L^2}^2
& \le \hat{\delta} \left( \mathbf{u}_n\left(t\right) \cdot \boldsymbol{\nabla} C_n\left(t\right), \Delta C_n\left(t\right) \right)\nonumber\\ 
& + A\left(\epsilon\right)\|C_n\left(t\right)\left(1-C_n\left(t\right)\right)\|_{L^2}^2 \quad a.e.\  on \ \left(0,T\right). \label{eq: L2 nabC 2}
\end{align}
Multiplying equation \eqref{eq: fdu} by $\alpha_j^n\left(t\right)$ and then summing from $j=1\ \text{to} \ n$ we obtain the following equation,
  \begin{align}
      \frac{1}{2} \frac{d}{dt} \left\| \mathbf{u}_n\left(t\right) \right\|_{L^2}^2 +\mu_e\left(\boldsymbol{\nabla} \mathbf{u}_n\left(t\right),\boldsymbol{\nabla} \mathbf{u}_n\left(t\right)\right) + \left( F\left(C_n\left(t\right)\right) 
 \mathbf{u}_n\left(t\right),\mathbf{u}_n\left(t\right)\right) = \left \langle \boldsymbol{\nabla} \cdot \mathbf{T}\left(C_n\left(t\right)\right),\mathbf{u}_n\left(t\right) \right \rangle \nonumber\\+\left(\mathbf f\left(t\right),\mathbf{u}_n\left(t\right)\right) \quad a.e. \ on\ \left(0,T\right). \label{L2u1}
        \end{align}
Utilizing the expression of Korteweg stress in \eqref{eq: Korteweg expression}, from equation \eqref{L2u1} we obtain,
\begin{align}
\frac{1}{2} \frac{d}{dt} \left\| \mathbf{u}_n\left(t\right) \right\|_{L^2}^2 + &\mu_e \left\| \boldsymbol{\nabla} \mathbf{u}_n\left(t\right) \right\|_{L^2}^2 + \|\sqrt{ F\left(C_n\left(t\right)\right) }\mathbf u_n\left(t\right) \|_{L^2}^2
\le -\hat{\delta} \left( \Delta C_n\left(t\right) \boldsymbol{\nabla} C_n\left(t\right), \mathbf{u}_n\left(t\right) \right) \nonumber \\
&\quad + \frac{1}{2}\left(\| \mathbf{f}\left(t\right)\|_{L^2}^2+\|\mathbf{u}_n\left(t\right) \|_{L^2}^2\right)\quad a.e. \ on \ \left(0,T\right), \label{eq: L2u2}
\end{align}
where $\hat \delta > 0$ is Korteweg parameter.
Adding inequality \eqref{eq: L2 nabC 2} and \eqref{eq: L2u2} we get the following:
\begin{align}
\frac{1}{2} \frac{d}{dt} \left\| \mathbf{u}_n\left(t\right) \right\|_{L^2}^2 
& + \hat{\delta} \frac{1}{2} \frac{d}{dt} \left\| \boldsymbol{\nabla} C_n\left(t\right) \right\|_{L^2}^2+\mu_e \left\| \boldsymbol{\nabla} \mathbf{u}_n\left(t\right) \right\|_{L^2}^2 \nonumber \\ 
& + \|\sqrt{ F\left(C_n\left(t\right)\right) }\mathbf u_n\left(t\right) \|_{L^2}^2 
+ \left(d\hat{\delta}-\epsilon\right)\|\Delta C_n\left(t\right)\|_{L^2}^2 \nonumber \\ 
& \leq \frac{1}{2}\left(\| \mathbf{f}\left(t\right)\|_{L^2}^2+\|\mathbf{u}_n\left(t\right) \|_{L^2}^2\right) \nonumber \\ 
& + A\left(\epsilon\right)\|C_n\left(t\right)\left(1-C_n\left(t\right)\right)\|_{L^2}^2 \quad a.e. \ on \ \left(0,T\right). \label{eq: L2C+u1}
\end{align}
Adding $\frac{1}{2}\hat{\delta}\left\|\boldsymbol{\nabla} C_n\left(t\right) \right\|_{L^2}^2  $ on the right hand side of inequality \eqref{eq: L2C+u1} and ignoring the positive terms of the left hand side and integrating from $0$ to $\tau$ for some $\tau \in \left(0,T\right)$ we get,
\begin{align}
\frac{1}{2} \left\| \mathbf{u}_n\left(\tau\right) \right\|_{L^2}^2 + \hat{\delta} \frac{1}{2} \left\| \boldsymbol{\nabla} C_n\left(\tau\right) \right\|_{L^2}^2
& \leq  \frac{1}{2}\left(\|\mathbf{u}_0\|_{L^2}^2+\hat \delta \|\boldsymbol{\nabla} C_0\|^2_{L^2}\right)+\int_0^\tau \frac{1}{2} \left\| \mathbf{f}\left(t\right) \right\|_{L^2}^2 \nonumber\\
& + \int_0^\tau \frac{1}{2}\left(\hat \delta \|\boldsymbol{\nabla} C_n\left(t\right)\|_{L^2}^2 + \| \mathbf{u}_n\left(t\right) \|_{L^2}^2\right) \nonumber \\ 
& + \int_0^\tau A\left(\epsilon\right)\|C_n\left(t\right)\left(1-C_n\left(t\right)\right)\|_{L^2}^2 \quad a.e. \ on \ \left(0,T\right). \label{eq: L2C+u}
\end{align}
Using Gr\"onwall’s inequality (Theorem \ref{Th: gronwall}) and Lemma \ref{lemma:L1} it can be demonstrated that the velocity field sequence $\mathbf{u}_n$ remains uniformly bounded in $L^\infty\left(0, T;\mathscr S_1\right)$, and the concentration sequence $C_n$ is uniformly bounded in the $L^\infty\left(0, T;\mathscr S_2\right)$ space.
Using the above result from \eqref{eq: L2C+u1} we can say $\mathbf u_n$ is uniformly bounded in $L^2\left(0,T;\mathscr V_1\right)$ and $C_n$ is uniformly bounded in $L^2\left(0,T;\mathscr V_2\right)$ space.
\end{proof}

\begin{lemma}
$ \displaystyle \frac{\partial C_n}{\partial t}$ is uniformly bounded in $L^{2}\left(0,T;L^2(\Omega)\right)$ for $N=2$, and $L^{2}\left(0,T;\mathscr V_2^*\right)$ for $N=3$. \label{lemma:L3}
\end{lemma}  
\begin{proof}
\textbf{For $\boldsymbol{N=2}$:}
Multiplying equation \eqref{eq: fdC} by \(  \displaystyle \frac{\partial \beta_{j}^{n} \left(t\right)}{\partial t} \), summing over \( j = 1 \) to \( n \), and applying the H\"older's inequality, we obtain 
\begin{align}
\left \| \frac{\partial C_n \left(t\right)}{\partial t} \right \|_{L^2}^2 
&\leq d \|\Delta C_n \left(t\right)\|_{L^2}\lf \lVert \frac{\partial C_n \left(t\right)}{\partial t} \rg\rVert_{L^2}  
+ \|\mathbf u_n \left(t\right) \|_{L^4}\| \boldsymbol{\nabla} C_n \left(t\right) \|_{L^4} \left \lV \frac{\partial C_n \left(t\right)}{\partial t} \right \rV_{L^{2}} \nonumber \\ 
&\quad + \kappa \|C_n \left(t\right)\left(1 - C_n \left(t\right)\right)\|_{L^2}\lf \| \frac{\partial C_n \left(t\right)}{\partial t}\rg \|_{L^2}, \quad a.e.\ on \ \left(0,T\right).
\label{eq: estimate of del c/del t 1st}
\end{align}
Using Gagliardo-Nirenberg inequality (Theorem \ref{Th: gagliardo}) in the second term of the right-hand side, we obtain the uniform boundedness of $C_n$ in $L^2(0,T;L^2(\Omega).$\\
\textbf{For $\boldsymbol{N=3}$:}
Let $v\in \mathscr V_2$ be any zero vector.
Then $v$ can be written as
\begin{align*}
    v = v^1 + v^2 \quad \text{where}\ v^1\in \left(\mathscr V_2\right)_n,\ v^2\in \left(\mathscr V_2\right)_n ^ \perp.
\end{align*}
Then from equation \eqref{eq: fdC} we get,
\begin{align}
    \left|\left\langle\frac{\partial C_n \left(t\right)}{\partial t},v\right\rangle\right|\leq \left|
    \left(\nabla C_n(t),\boldsymbol\nabla v^1\right)\right| + \left|
    \left( C_n(t)\mathbf{u}_n,\boldsymbol\nabla v^1\right)\right| + \kappa\left|
    \left( C_n(1-C_n),v^1\right)\right|
\end{align}
Now applying Cauchy-Schwarz inequality to the terms of the right-hand side and taking $supremum$ over all $ v \in \mathscr V_2$ with $\|v\|_{\mathscr V_2} \le1$, we get,
 \begin{align}
     \left\|\frac{\partial C_n \left(t\right)}{\partial t}\right\|_{\mathscr V_2^*} &\leq \|\boldsymbol\nabla C_n(t)\|_{L^2} + \|C_n(t)\|_{L^4}\|u_n(t)\|_{L^4} + \kappa \|C_n(t)(1-C_n(t))\|_{L^2}\nonumber
     \\
     &\leq \|\boldsymbol\nabla C_n(t)\|_{L^2} + \|C_n(t)\|_{H^1}\|\boldsymbol\nabla u_n(t)\|_{L^2} + \kappa \|C_n(t)(1-C_n(t))\|_{L^2} \label{eq: estimate of del c/del t 2nd}
      \end{align}
The first two terms in the right-hand side of inequality \eqref{eq: estimate of del c/del t 2nd} can be effectively controlled using Lemma \ref{lemma:L2}, while the last term is bounded via Remark \ref{rem: bound of C(1-C)}. Consequently, we establish the uniform boundedness of $ \displaystyle  \frac{\partial C_n}{\partial t} $ in $L^2\left(0,T;\mathscr V_2^*\right)$. 
\end{proof}
\begin{lemma}
	$ \displaystyle \frac{\partial \mathbf{u}_n}{\partial t}$ is uniformly bounded in $L^{4/N}\left(0,T;\mathscr V_1^*\right)$ for $N=2,3$. \label{lemma:L4}
\end{lemma}

\begin{proof}
		Let $\mathbf w\in \mathscr V_1$ be any non-zero vector. 
  Then $\mathbf w$ can be written as
  \begin{align*}
      \mathbf w=\mathbf w^1+\mathbf w^2\quad \text{where}\ \mathbf w^1\in \left(\mathscr V_1\right)_n,\ \mathbf w^2\in \left(\mathscr V_1\right)_n ^ \perp.
 \end{align*}
 It can be shown that
 \begin{align*}
     \left \langle \frac{\partial \mathbf{u}_n\left(t\right)}{\partial t},\mathbf w^2 \right \rangle = 0.
 \end{align*}
 Then from equation \eqref{eq: fdu} we get,
 \begin{align}
     \left \lvert \left \langle \frac{\partial \mathbf{u}_n\left(t\right)}{\partial t},\mathbf w \right \rangle \right \rvert 
     & = \lf| \left \langle \frac{\partial \mathbf{u}_n\left(t\right)}{\partial t},\mathbf w^1 \right \rangle \rg| \nonumber \\ 
     & = \left| \mu_e\left(\boldsymbol{\nabla} \mathbf{u}_n\left(t\right), \boldsymbol{\nabla} \mathbf w^1\right) +  \left(F\left(C_n\left(t\right)\right) \mathbf{u}_n\left(t\right),\mathbf w^1\right)-\left \langle \boldsymbol{\nabla} \cdot \mathbf{T}\left(C_n\left(t\right)\right) , \mathbf w^1 \right \rangle  + \left(\mathbf f\left(t\right),\mathbf w^1\right) \right| \quad  a.e. \ on \ \left(0,T\right) \nonumber\\ 
     & \le \left| \mu_e\left(\boldsymbol{\nabla} \mathbf{u}_n\left(t\right), \boldsymbol{\nabla} \mathbf w^1\right)\right| + \left|\left(  F\left(C_n\left(t\right)\right) \mathbf{u}_n\left(t\right),\mathbf w^1\right)\right|+\left|\left \langle \boldsymbol{\nabla} \cdot \mathbf{T}\left(C_n\left(t\right)\right) , \mathbf w^1 \right \rangle\right| \nonumber\\ 
     & \qquad\qquad\qquad\qquad + \left|\left(\mathbf f\left(t\right),\mathbf w^1\right)\right| \quad  a.e. \ on \ \left(0,T\right). 
 \end{align} 
 The term $\left \lvert \left(  F\left(C_n\left(t\right)\right) \mathbf{u}_n\left(t\right),\mathbf w^1\right) \right \rvert$ can be estimated as,
\begin{align}
     \left\lvert \left(  F\left(C_n\left(t\right)\right) \mathbf{u}_n\left(t\right),\mathbf w^1\right) \right\rvert &= \left \lvert \left( \mathbf{u}_n\left(t\right), \left(  F\left(C_n\left(t\right)\right) \mathbf w^1\right)\right) \right \rvert \nonumber\\& \le \|\mathbf u_n\left(t\right)\|_{L^2}\|{F\left(C_n\left(t\right)\right)}\|_{L^4}\|\mathbf w^1\|_{L^4} \nonumber\\
     &\leq M \lVert {F \left( C_n \left( t \right) \right)} \rVert_{H^1} \lVert \mathbf w^1 \rVert_{H^1} \quad a.e. \ on \ \left( 0, T \right). \label{eq: estimate of F(C_n)} 
\end{align}
Now applying Cauchy-Schwarz inequality in the rest of the terms of the right-hand side and taking $supremum$ over all $\mathbf w \in \mathscr V_1$ with $\|\mathbf w\|_{H^1} \le1$, we get,
\begin{align}
     \left \lVert \frac{\partial \mathbf{u}_n\left(t\right)}{\partial t} \right\rVert_{\mathscr V_1^*} \le \mu_e\|\boldsymbol{\nabla} \mathbf{u}_n\left(t\right)\|_{L^2} 
     & + M\|{ F\left(C_n\left(t\right)\right)}\|_{H^1} + \| \boldsymbol{\nabla} \cdot \mathbf{T}\left(C_n\left(t\right)\right) \|_{\mathscr V_1^*} \nonumber \\
     & + \| \mathbf f\left(t\right)\|_{L^2}  \quad  a.e. \ on \ \left(0,T\right)\label{eq: L51}. 
\end{align}
Here, $\nabla \cdot \mathbf{T}(C)$ is a sum of expressions of the form $\lambda D_i(D_j C \, D_k C)$. We then have the following estimate
\begin{align}\label{est: korteweg stress}
\| D_i(D_j C \, D_k C) \|_{\mathscr  V_1^*} & \leq \| D_j C \, D_k C \|_{L^2} \leq \| D_j C \|_{L^4} \, \| D_k C \|_{L^4}  \nonumber\\
\leq & M  \lV D_j C\rV^{(4-N)/4}_{L^2}  \lV D_k C\rV^{(4-N)/4}_{L^2} \| D_j C \|_{H^1}^{N/4} \, \| D_k C \|_{H^1}^{N/4} \quad\text{for}\ N = 2,3 .
\end{align}
The remaining terms on the right-hand side are uniformly bounded as a consequence of Lemma \ref{lemma:L2}.  Hence, $ \displaystyle \frac{\partial \mathbf u_n}{\partial t}$ is uniformly bounded in $L^{4/N}\left(0,T;\mathscr V_1^*\right)$ for $N=2,3$.
\end{proof}

\subsubsection{Existence}

Using Lemma \ref{lemma:L1} -- Lemma \ref{lemma:L4} and following the analysis of \citet{kundu2024existence, kostin2003modelling}, we obtain uniformly convergent subsequence of $u_n$, $C_n$ (again using the same notation),
	\begin{align*}
		&\quad C_n  \rightharpoonup C\  in\ L^2\left(0,T;\mathscr V_2\right),   \\ 
		&\quad C_n \stackrel{*}{\rightharpoonup} C \ in \ L^\infty\left(0,T; \mathscr S_2\right),  \\
		&\quad \frac{\partial C_n}{\partial t} \rightharpoonup \frac{\partial C}{\partial t} \ in \ L^2\left(0,T;\mathscr V_2^*\right),   \\ 
		&\quad \mathbf{u}_n \rightharpoonup \mathbf{u} \ in \ L^2\left(0,T;\mathscr V_1\right), \\	
		&\quad \mathbf{u}_n \stackrel{*}{\rightharpoonup} \mathbf{u} \ in \ L^\infty \left(0,T; \mathscr S_1\right),    \\	
		&\quad \frac{\partial \mathbf u_n}{\partial t} \rightharpoonup \frac{\partial \mathbf u}{\partial t} \ in \ L^{4/N}\left(0,T;\mathscr V_1^*\right)\quad\text{for}\ N = 2,3 . \end{align*}
	Applying Aubin-Lions lemma (\emph{\citep[Corollary 4]{simon1986compact}}) we get strongly convergent subsequences of $(\mathbf u_n,C_n)$ in $L^2\left(0,T; \mathscr S_1\right) \times L^2\left(0,T; \mathscr S_2\right)$. Again from Aubin-Lions (\emph{\citep[Corollary 4]{simon1986compact}}) we get compact embeddings of these sequence of Galerkin ($\mathbf{u}_n,C_n$) in $\left(\mathcal{C}\left(0,T; \mathscr S_1\right),\mathcal{C}\left(0,T; \mathscr S_2\right)\right)$.\\
	For all $m \le n$ we have
	\begin{align}
	\left(\frac{\partial C_n\left(t\right)}{\partial t},B\right)
    & + d\left(\boldsymbol{\nabla} C_n\left(t\right), \boldsymbol{\nabla} B\right) + \left(\mathbf{u}_n\left(t\right) \cdot \boldsymbol{\nabla} C_n\left(t\right),B\right) \nonumber \\ 
    & + \kappa \left(C_n\left(t\right)\left(1-C_n\left(t\right)\right),B\right) = 0 \quad \forall B \in \left(\mathscr V_2\right)_m \ a.e. \ on \ \left(0,T\right) \label{eq: C_n pass limit}.
	\end{align}
Now,
\begin{align}
\Big(\left(\mathbf{u}_n\left(t\right) \cdot \boldsymbol{\nabla} C_n\left(t\right)-\mathbf{u}\left(t\right) \cdot \boldsymbol{\nabla} C\left(t\right)\right), B\Big)
& \le \| \boldsymbol{\nabla} C_n\left(t\right) - \boldsymbol{\nabla} C\left(t\right)\|_{L^2}\|\mathbf{u}_n\left(t\right)B\|_{L^2} \nonumber \\ 
& +\| \mathbf{u}_n\left(t\right)-\mathbf{u}\left(t\right)\|_{L^2}\|\boldsymbol{\nabla} C\left(t\right)B\|_{L^2} \nonumber\\  
& \qquad\qquad\qquad \forall \ B \in \left(\mathscr V_2\right)_m .\label{eq: estimate of u.grad C}
\end{align}
By strong convergence of $\mathbf{u}_n$ and $C_n$ in $L^2\left(0,T; \mathscr S_1\right)$ and $L^2\left(0,T; \mathscr S_2\right)$, respectively, we get right hand side of inequality \eqref{eq: estimate of u.grad C} tends to 0. Again using the strong convergence of $C_n$ in $L^2\left(0,T; \mathscr S_2\right)$, we get 
$$C_n\left(1-C_n\right) \to C\left(1-C\right) \ \text{in} \ L^2\left(0,T; L^2(\Omega)\right).$$ 
Noting that $\displaystyle \bigcup_{m \ge 1} \left(\mathscr V_2\right)_m$ is dense in $\mathscr V_2$ and taking $n \to \infty$ in equation \eqref{eq: C_n pass limit} we get,
\begin{align}
	 \left(\frac{\partial C\left(t\right)}{\partial t},B\right)
     & + d\left(\boldsymbol{\nabla} C\left(t\right), \boldsymbol{\nabla} B\right) + \left(\mathbf{u}\left(t\right) \cdot \boldsymbol{\nabla} C\left(t\right),B\right) \nonumber \\ 
     & + \kappa \left(C\left(t\right)\left(1-C\left(t\right)\right),B\right)=0 \quad \forall B \in \mathscr V_2 \ a.e. \ on \ \left(0,T\right). \label{eq: existence of C}
\end{align} 
Again for all $m\le n$ we have 
\begin{align}
	\left \langle  \frac{\partial \mathbf{u}_n\left(t\right)}{\partial t},\mathbf{v} \right \rangle + \mu_e\left(\boldsymbol{\nabla} \mathbf{u}_n\left(t\right), \boldsymbol{\nabla} \mathbf{v}\right)
    & + \left(F\left(C_n\left(t\right)\right) \mathbf{u}_n\left(t\right),\mathbf{v}\right) -\left \langle \boldsymbol{\nabla} \cdot \mathbf{T}\left(C_n\left(t\right)\right) , \mathbf{v} \right \rangle \nonumber \\ 
    & - \left(\mathbf f\left(t\right),\mathbf v\right) = 0 \quad \forall \mathbf{v} \in \left(\mathscr V_1\right)_m \ a.e. \ on \ \left(0,T\right). \label{eq: u_n pass limit}
\end{align}
Now since $\mathbf{u}_n \rightharpoonup \mathbf{u}$ in $L^2\left(0,T;\mathscr V_1\right)$, $ \displaystyle \frac{\partial \mathbf u_n}{\partial t} \rightharpoonup \frac{\partial \mathbf u}{\partial t} $ in $ L^{4/N}\left(0,T;\mathscr V_1^*\right) $ we can pass the limit in the terms \eqref{eq: u_n pass limit}$_1$ and \eqref{eq: u_n pass limit}$_2$. \\Using the strong convergence of $C_n$,\ $\mathbf u_n$ and assumption \eqref{assump:A1} we obtain,
\begin{align}
    \left(F\left(C_n\left(t\right)\right) \mathbf u_n\left(t\right),\mathbf v\right)\to\left(F\left(C\left(t\right)\right)\mathbf u\left(t\right),\mathbf v\right) \quad \forall v \in \left(\mathscr V_1\right)_m, \ \forall m \le n \ a.e.  \ on \ \left(0,T\right) \label{eq: Fu}. 
    \end{align}

For the term corresponding to $\nabla \cdot \mathbf{T}(C)$ we notice from equation \eqref{est: korteweg stress}, that $D_j C_n D_k C_n$ is bounded in $L^2(0, T; L^2(\Omega))$, so that some subsequence
converges weakly in $L^2(0, T; L^2(\Omega))$ to some function $\Xi_{jk}$, where $\Xi_{jk} \in L^2(0, T; L^2(\Omega))$. Using the strong convergence of $C_n$ in
the space $L^2(0, T; \mathscr S_2)$, we immediately infer that $\Xi_{jk} = D_j C D_k C$.

Now using the fact that $\cup_{m\ge1}\left(\mathscr V_1\right)_m$ is dense in $\mathscr V_1$ and passing the limit in equation \eqref{eq: u_n pass limit} we get,
\begin{align}
	\left \langle  \frac{\partial \mathbf{u}\left(t\right)}{\partial t},\mathbf{v} \right \rangle + \mu_e\left(\boldsymbol{\nabla} \mathbf{u}\left(t\right), \boldsymbol{\nabla} \mathbf{v}\right)
    & + \left(F\left(C\left(t\right)\right) \mathbf{u}\left(t\right),\mathbf{v}\right) -\left \langle \boldsymbol{\nabla} \cdot \mathbf{T}\left(C\left(t\right)\right) , \mathbf{v} \right \rangle \nonumber \\ 
    & - \left(\mathbf f\left(t\right),\mathbf v\right) = 0 \quad \forall \mathbf{v} \in \mathscr V_1\ a.e. \ on \ \left(0,T\right) \label{eq: existence u}.
\end{align}
Again, to demonstrate the existence of the pressure, we apply de Rham's theorem (\emph{\citep[Theorem IV.2.4]{boyer2012mathematical}}) and proceed analogously to the method outlined in \citep{kundu2024existence}.

\subsection{Proof of Theorem \ref{th: bddness & blow_up}}

\subsubsection{Global existence of solution}
We first show that if $C_0\geq0$ then $C\geq 0.$
Take $C^-=\max\{0,-C\}$ as test function in equation \eqref{var C} and obtain the following equation:
\begin{align}
    \frac{1}{2}\frac{d}{dt}\|C^-(t)\|_{L^2}^2+d\|\boldsymbol\nabla C^-(t)\|_{L^2}^2 =- \kappa \|C^-(t)\|_{L^2}^2-\kappa\left(C^2(t),C^-(t)\right)
\end{align}
The right-hand side of the above equation is non-positive. This gives us,
\begin{align}
     \frac{1}{2}\frac{d}{dt}\|C^-(t)\|_{L^2}^2\le 0. \label{eq: C^- 2nd}
\end{align}
Integrating inequality \eqref{eq: C^- 2nd} from $0$ to $\tau$ for any $\tau \in (0,T)$ we get,
\begin{align}
    \|C^-(\tau)\|_{L^2}^2\le \| C^-(0) \|_{L^2}^2 = 0. 
\end{align}
This implies $C\ge0$ almost everywhere on $\Omega\times(0,T)$.  

We now show that if $0\leq C_0\leq \gamma_1 \leq1$ then $C\leq \gamma_1$. For that we take $\phi = \max\{C-\gamma_1,0\}$ as a test function in equation \eqref{var C}, and obtain the following equation:
\begin{align}
    \frac{1}{2}\frac{d}{dt}\|\phi(t)\|_{L^2}^2+d\|\boldsymbol\nabla \phi (t)\|_{L^2}^2 &= \kappa \left ( C (t) (C(t) -1), \phi(t)\right) \nonumber \\
    & = \kappa\|\phi(t)\|_{L^3}^3+\kappa \gamma_1\|\phi(t)\|_{L^2}^2 +\kappa (\gamma_1 - 1)\|\phi(t)\|_{L^2}^2 + \kappa(\gamma_1^2-\gamma_1)\int_\Omega \phi(t)d\Omega
\end{align}
The last two terms of the above inequality are non-positive, which gives us,
\begin{align}
    \frac{1}{2}\frac{d}{dt}\|\phi(t)\|_{L^2}^2+d\|\boldsymbol\nabla \phi (t)\|_{L^2}^2 &\leq \kappa\|\phi(t)\|_{L^3}^3+\kappa \gamma_1\|\phi(t)\|_{L^2}^2 \label{eq: C<1 0th}\\
    &\leq  M \|\phi(t)\|^3_{L^{6-N}} + \kappa \gamma_1 \|\phi(t)\|^2_{L^2}\quad \text{for} \ N=2,3.\label{eq: C<1 1st}
\end{align}
Applying Gagliardo-Nirenberg inequality (Theorem \ref{Th: gagliardo}) to the first term of right hand side, we get,
\begin{align}
    \frac{1}{2}\frac{d}{dt}\|\phi(t)\|_{L^2}^2+d\|\boldsymbol\nabla \phi\|_{L^2}^2
    &\leq M \|\phi\|^{3/2}_{L^2}\|\phi\|^{3/2}_{H^1} + \kappa \gamma_1 \|\phi(t)\|^2_{L^2} .\label{eq: C<1 1st}
\end{align}
Using Young's inequality and  using the fact that $C\in L^\infty(0,T;H^1)$, we obtain,
\begin{align}
    \frac{1}{2}\frac{d}{dt}\|\phi(t)\|_{L^2}^2\leq A(\epsilon)\|\phi\|^{2}_{L^2}.\label{eq: C<1 2nd}
\end{align}
Applying Gr\"onwall's inequality, we get
\begin{align}
\|\phi(t)\|_{L^2}^2 \leq 0.  
\end{align}
This gives $C\leq \gamma_1$ almost everywhere in $\Omega \times (0,T)$. 

Now, we prove that the solution exists for any time $T>0.$ The existence of a solution up to a finite time is coming due to Lemma \ref{lemma:L1}, where the time interval is getting restricted due to inequality \eqref{uniform boundedness of C_n 3rd}. 
We now show that the solution remains uniformly bounded for any $T>0$. For that we take $C(t)$ as a test function in equation \eqref{var C}:
	\begin{align}
    \frac{1}{2}\frac{d}{dt}\|C\left(t\right)\|_{L^2}^2&=- \kappa \left(C(t)\left(1-C(t)\right),C(t)\right) \le -\kappa\left(1 - \gamma_1\right)\|C(t)\|^2_{L^2} \quad a.e.\  on \ \left(0,T\right). \label{uniform boundedness C for all time 1st }
	\end{align}
    This gives
    \begin{align}
        \|C(t)\|^2_{L^2}\leq \exp{\left(-\kappa\left(1 - \gamma_1\right)t\right)}\|C_0\|^2_{L^2}.  \label{uniform boundedness C for all time 2nd}
    \end{align}
This gives the uniform boundedness of $C$ in $L^\infty(0,T;L^2(\Omega)$ for any $T>0.$ Also, inequality \eqref{uniform boundedness C for all time 2nd} gives that the solute concentration converges exponentially to $0$ as $t\to \infty$ if $C_0 \leq \gamma_1 < 1.$

\subsubsection{Finite time blow up}
We now show that if $C_0\geq \gamma_2 >1$ then $C\geq \gamma_2$. For that we take $\psi = \min\{C-\gamma_2,0\}$ as a test function in equation \eqref{var C}, and obtain the following equation:
\begin{align}
    \frac{1}{2}\frac{d}{dt}\|\psi(t)\|_{L^2}^2+d\|\boldsymbol\nabla \psi (t)\|_{L^2}^2 &= \kappa \left ( C (t) (C(t) -1), \psi(t)\right) \nonumber \\
    & = \|\psi(t)\|_{L^3}^3 + (2\gamma_2 - 1)\|\psi(t)\|_{L^2}^2 + (\gamma_2^2-\gamma_2)\int_\Omega \psi(t)d\Omega
\end{align}
The last part of the above inequality is non-positive due to the construction of the function $\psi$, which gives us,
\begin{align}
    \frac{1}{2}\frac{d}{dt}\|\psi(t)\|_{L^2}^2+d\|\boldsymbol\nabla \psi (t)\|_{L^2}^2 \leq \|\psi(t)\|_{L^3}^3 + (2\gamma_2 - 1)\|\psi(t)\|_{L^2}^2.
\end{align}
We have now arrived at a similar expression as in equation \eqref{eq: C<1 0th}. So, proceeding similarly, we obtain
\begin{align}
    \|\psi (t)\|_{L^2} = 0.
\end{align}
Which gives us $C\geq \gamma_2$ almost everywhere in $\Omega\times(0,T)$. 
Next, we prove that the solute concentration blows up in a finite time. For that we define $\mathscr C(t)= \int_\Omega C(t) d\Omega$. 

Taking $B = 1 $ as a test function in equation \eqref{var C}, we get
\begin{align}
 \frac{d \mathscr C(t)}{d t} &= - \kappa \int_\Omega C(t) (1-C(t)) d\Omega \nonumber
 \\
 &\geq \frac{\kappa}{|\Omega|}\left(\mathscr C^2(t) - |\Omega| \mathscr C(t)\right).
\end{align}
Again upon integration from $0$ to $t$, we get
\begin{align}
    \mathscr C(t) \geq \frac{|\Omega|\mathscr C(0)}{\mathscr C(0)-e^{\kappa t}\left(\mathscr C(0) -|\Omega|\right)}
\end{align}
The right-hand side of the above inequality tends to infinity if $\displaystyle t\to T_*^-$, where 
\begin{align}
T_* =  \left(\frac{1}{\kappa}\ln{\frac{\mathscr C(0)}{\mathscr C(0) - |\Omega|}}\right).
\end{align}Hence we have
\begin{align}
    \lim_{t\to T_*^-} \|C(t)\|_{L^1} = +\infty. \label{eq: blowup time}
\end{align}

\subsection{Proof of Theorem \ref{th: Uniqueness}}
Let $\left(\mathbf{u}_1,C_1\right)$ and  $\left(\mathbf{u}_2,C_2\right)$ be two solutions. We define, $\mathbf{u}=\mathbf{u}_1-\mathbf{u}_2$, $C=C_1-C_2$. Then we have the following system of equations,
\begin{align}
    \left(\frac{\partial C\left(t\right)}{\partial t},B\right)
    & + d\left(\boldsymbol{\nabla} C\left(t\right),\boldsymbol{\nabla} B\right)+\left(\mathbf{u}_1 \left(t\right)\cdot\boldsymbol{\nabla} C_1\left(t\right)-\mathbf{u}_2 \left(t\right)\cdot \boldsymbol{\nabla} C_2\left(t\right),B\right) + \kappa \left(C\left(t\right), B\right) \nonumber \\ 
    & -\kappa \left(\left(C_1\left(t\right) + C_2\left(t\right)\right)C\left(t\right), B\right) = 0 \quad  \forall B \in \mathscr V_2\ a.e.\ on\ \left(0,T\right),\label{eq: uC1} \\
   \left \langle\frac{\partial\mathbf{u}\left(t\right)}{\partial t},\mathbf{v}\right \rangle 
   & + \mu _e\left(\boldsymbol{\nabla}\mathbf{u}\left(t\right),\boldsymbol{\nabla} \mathbf v\right) + \left(F\left(C_1\left(t\right)\right) \mathbf u_1\left(t\right)-F\left(C_1\left(t\right)\right) \mathbf u_2\left(t\right),\mathbf v\right) \nonumber \\ 
   & - \left \langle\boldsymbol{\nabla}\left(T\left(C_1\left(t\right)\right)-T\left(C_2\left(t\right)\right)\right),\mathbf v\right \rangle = 0 \quad \forall \mathbf{v} \in \mathscr V_1 \ a.e.\ on\ \left(0,T\right). \label{eq: uu1}
\end{align}
Putting $B=C\left(t\right)$ in \eqref{eq: uC1} we get,
\begin{align}
    \frac{1}{2} \frac{d}{dt}\|C\left(t\right)\|^2_{L^2}+d\|\boldsymbol{\nabla} C\left(t\right)\|^2_{L^2}+&\left(\mathbf u_1\left(t\right)\cdot\boldsymbol{\nabla} C_1\left(t\right)-\mathbf u_2\left(t\right)\cdot\boldsymbol{\nabla} C_2\left(t\right),C\left(t\right)\right)+\kappa \|C\left(t\right)\|^2_{L^2}\nonumber\\  =&\kappa \left(\left(C_1\left(t\right)+C_2\left(t\right)\right)C\left(t\right),C\left(t\right)\right) \quad a.e.\ on\ \left(0,T\right). \label{eq: uC2}
\end{align}
Now,
\begin{align}
\left(\mathbf u_1\left(t\right)\cdot\boldsymbol{\nabla} C_1\left(t\right)-\mathbf u_2\left(t\right)\cdot\boldsymbol{\nabla} C_2\left(t\right),C\left(t\right)\right)=\left(\mathbf u_1\left(t\right)\cdot\boldsymbol{\nabla} C\left(t\right),C\left(t\right)\right)+\left(\mathbf u\left(t\right) \cdot \boldsymbol{\nabla} C_2\left(t\right),C\left(t\right)\right).
\end{align}
For the term $\left(\mathbf u_1\left(t\right)\cdot\boldsymbol{\nabla} C\left(t\right),C\left(t\right)\right)$ we have the following estimate,
\begin{align}
    |\left(\mathbf u_1\left(t\right)\cdot\boldsymbol{\nabla} C\left(t\right),C\left(t\right)\right)|\le&\|\mathbf u_1\left(t\right) \cdot \boldsymbol{\nabla} C\left(t\right)\|_{L^2}\|C\left(t\right)\|_{L^2}\nonumber\\
    \le & \|\boldsymbol{\nabla} C\left(t\right)\|_{L^4}\|\mathbf u_1\left(t\right)\|_{L^4}\|C\left(t\right)\|_{L^2} \nonumber \\
    \le & M_1\| C\left(t\right)\|^{\frac{1}{2}}_{L^2}\|\Delta C\left(t\right)\|^{\frac{1}{2}}_{L^2}\|\mathbf u_1\left(t\right)\|^{\frac{1}{2}}_{L^2}\|\boldsymbol{\nabla}\mathbf u_1\left(t\right)\|^{\frac{1}{2}}_{L^2}\|C\left(t\right)\|_{L^2} \nonumber\\
    \le & M\|\Delta C\left(t\right)\|^{\frac{1}{2}}_{L^2}\|\boldsymbol{\nabla}\mathbf u_1\left(t\right)\|^{\frac{1}{2}}_{L^2} \|\boldsymbol{\nabla} C\left(t\right)\|^{\frac{1}{2}}_{L^2}\quad a.e.\ on\ \left(0,T\right) \nonumber \\ 
    & \qquad \qquad \qquad \qquad \text{(using Lemmas \ref{lemma:L1} and  \ref{lemma:L2})} \\
    \le & \epsilon \|\Delta C\left(t\right)\|^2_{L^2}+A\left(\epsilon\right) \|\boldsymbol{\nabla}\mathbf u_1\left(t\right)\|^2_{L^2} \|\boldsymbol{\nabla} C\left(t\right)\|^2_{L^2}\quad a.e.\ on\ \left(0,T\right).
\end{align}
Again, for the term $\left(\mathbf u\left(t\right) \cdot \boldsymbol{\nabla} C_2\left(t\right),C\left(t\right)\right)$ we can show,
\begin{align}
   \left\lvert  \left(\mathbf u\left(t\right) \cdot \boldsymbol{\nabla} C_2\left(t\right),C\left(t\right)\right) \right\rvert\le & \epsilon \|\Delta C\left(t\right)\|^2_{L^2}+A\left(\epsilon\right) \|\boldsymbol{\nabla}\mathbf u\left(t\right)\|^2_{L^2} \|\mathbf u\left(t\right)\|^2_{L^2}\quad a.e.\ on\ \left(0,T\right).
\end{align}
and for the term $\kappa \left(\left(C_1\left(t\right)+C_2\left(t\right)\right)C\left(t\right),C\left(t\right)\right)$ we can show,
\begin{align}
   | \kappa \left(\left(C_1\left(t\right)+C_2\left(t\right)\right)C\left(t\right),C\left(t\right)\right)|\le M \|C\left(t\right)\|^2_{L^2} \quad a.e.\ on\ \left(0,T\right) .
\end{align}
Thus using these inequalities in equation \eqref{eq: uC2} we get,
\begin{align}
     \frac{1}{2} \frac{d}{dt}\|C\left(t\right)\|^2_{L^2}+d\|\boldsymbol{\nabla} C\left(t\right)\|^2_{L^2}+\kappa
     \| C\left(t\right)\|^2_{L^2}\le 2 \epsilon \|\Delta C\left(t\right)\|^2_{L^2}+A\left(\epsilon\right) \|\boldsymbol{\nabla}\mathbf u_1\left(t\right)\|^2_{L^2} \|\boldsymbol{\nabla} C\left(t\right)\|^2_{L^2} \nonumber\\+ A\left(\epsilon\right) \|\boldsymbol{\nabla}\mathbf u\left(t\right)\|^2_{L^2} \|\mathbf u\left(t\right)\|^2_{L^2} +M \|C\left(t\right)\|^2_{L^2}  \quad a.e.\ on\ \left(0,T\right). \label{eq: uC3}
\end{align}
Again, taking $B=-\Delta C\left(t\right)$ in equation \eqref{eq: uC1} and using integration by parts on the first and second term we get,
\begin{align}
    \frac{ 1}{2} \frac{d}{dt}\|\boldsymbol{\nabla} C\left(t\right)\|^2_{L^2}+d    \| \Delta C\left(t\right)\|^2_{L^2} 
    & + \kappa \|\boldsymbol{\nabla} C\left(t\right)\|^2_{L^2} \nonumber \\ 
    & = \left(\mathbf u_1\left(t\right)\cdot\boldsymbol{\nabla} C_1\left(t\right)-\mathbf u_2\left(t\right)\cdot\boldsymbol{\nabla} C_2\left(t\right),\Delta C\left(t\right)\right) \nonumber \\  
    & \quad - \kappa \left(\left(C_1\left(t\right)+C_2\left(t\right)\right)C\left(t\right),\Delta C\left(t\right)\right) \quad a.e.\ on\ \left(0,T\right)\nonumber \\
    & \le \left(\mathbf u_1\left(t\right)\cdot\boldsymbol{\nabla} C\left(t\right),\Delta C\left(t\right)\right) + \left(\mathbf u\left(t\right) \cdot \boldsymbol{\nabla} C_2\left(t\right) ,\Delta C\left(t\right)\right) \nonumber \\ 
    & \quad - \kappa \left(\left(C_1\left(t\right)+C_2\left(t\right)\right)C\left(t\right),\Delta C\left(t\right)\right) \quad a.e.\ on\ \left(0,T\right) \label{eq: unabC1}.
\end{align}
For the term $\left(\mathbf u_1\left(t\right)\cdot\boldsymbol{\nabla} C\left(t\right),\Delta C\left(t\right)\right)$ we can show the following estimate,
\begin{align}
   | \left( \mathbf u_1\left(t\right)\cdot\boldsymbol{\nabla} C\left(t\right),\Delta C\left(t\right) \right)|\le & \|\mathbf u_1\left(t\right)\|_{L^4}\|\boldsymbol{\nabla} C\left(t\right)\|_{L^4}\|\Delta C\left(t\right)\|_{L^2} \nonumber \\
    \le& M \| \Delta C\left(t\right)\|^\frac{3}{2}_{L^2} \|\boldsymbol{\nabla} C\left(t\right)\|^\frac{1}{2}_{L^2}\|\mathbf u_1\|^\frac{1}{2}_{L^2}\|\boldsymbol{\nabla} \mathbf u_1\|^\frac{1}{2}_{L^2} \quad a.e.\ on\ \left(0,T\right)  \nonumber \\
    \le& \epsilon \| \Delta C\left(t\right)\|^{2}_{L^2} + A\left(\epsilon\right) \|\boldsymbol{\nabla} C\left(t\right)\|^{2}_{L^2} \|\boldsymbol{\nabla} \mathbf u_1\|^{2}_{L^2} \quad a.e.\ on\ \left(0,T\right), \label{eq: u.nab C, delta C}
\end{align}
and for the term $\kappa \left(\left(C_1\left(t\right)+C_2\left(t\right)\right)C\left(t\right),\Delta C\left(t\right)\right)$ we have,
\begin{align}
    \left | \kappa \left(\left(C_1\left(t\right)+C_2\left(t\right)\right)C\left(t\right),\Delta C\left(t\right)\right) \right | 
    & \le \epsilon \|\boldsymbol{\nabla} C\left(t\right)\|_{L^2}^2 + 
    \epsilon \|\Delta C\left(t\right)\|_{L^2}^2 \nonumber \\ 
    & \quad + A\left(\epsilon\right) \|C\left(t\right)\|_{L^2}^2 \quad a.e. \ on \ \left(0,T\right).\label{eq: estimate (c_1+c_2)C. lap C}
\end{align}
Substituting inequality \eqref{eq: u.nab C, delta C} and \eqref{eq: estimate (c_1+c_2)C. lap C} into inequality \eqref{eq: unabC1}, we derive,
\begin{align}
    \frac{  1}{2}\frac{d}{dt}\|\boldsymbol{\nabla} C\left(t\right)\|^2_{L^2}
    & +d \| \Delta C\left(t\right)\|^2_{L^2}+\kappa \|\boldsymbol{\nabla} C\left(t\right)\|_{L^2}^2 \nonumber \\
    & \le 2 \epsilon \| \Delta C\left(t\right)\|^{2}_{L^2} + \epsilon \|\boldsymbol{\nabla} C\left(t\right)\|_{L^2}^2 + A\left(\epsilon\right) \|\boldsymbol{\nabla} C\left(t\right)\|^{2}_{L^2} \|\boldsymbol{\nabla} \mathbf u_1(t)\|^{2}_{L^2} \nonumber \\
    & \qquad + \left(\mathbf u\left(t\right) \cdot \boldsymbol{\nabla} C_2\left(t\right) ,\Delta C\left(t\right)\right) + A\left(\epsilon\right) \|C\left(t\right)\|_{L^2}^2
    \quad a.e.\ on\ \left(0,T\right) \label{eq: unabC2}.
\end{align}
By setting $\mathbf{v} = \mathbf{u}\left(t\right)$ in equation \eqref{eq: uu1}, we obtain,
\begin{align}
   \frac{ 1}{2}\frac{d}{dt}\| \mathbf  u\left(t\right)\|^2_{L^2}+\mu _e\|\boldsymbol{\nabla}\mathbf{u}\left(t\right)\|^2_{L^2}+\left(F\left(C_1\left(t\right)\right) \mathbf u_1\left(t\right)-F\left(C_2\left(t\right)\right) \mathbf u_2\left(t\right),\mathbf u\left(t\right)\right)\nonumber\\- \left \langle\boldsymbol{\nabla}\cdot\left(T\left(C_1\left(t\right)\right)-T\left(C_2\left(t\right)\right)\right),\mathbf   u\left(t\right)\right \rangle=0  \quad a.e.\ on\ \left(0,T\right).
\end{align}
By applying equation \eqref{eq: Korteweg expression} and rearranging the terms, we get,
\begin{align}
   \frac{ 1}{2}\frac{d}{dt}\| \mathbf  u\left(t\right)\|^2_{L^2} 
   & + \mu _e \|\boldsymbol{\nabla}\mathbf{u}\left(t\right)\|^2_{L^2}+\left(F\left(C_1\left(t\right)\right)\mathbf u\left(t\right),\mathbf u\left(t\right)\right) \nonumber \\ 
   & + \left(F\left(C_1\left(t\right)\right)-F\left(C_2\left(t\right)\right)\mathbf u_2\left(t\right),
   \mathbf u\left(t\right)\right) + \hat \delta \left(\boldsymbol{\nabla} C\left(t\right) \Delta C_1\left(t\right),\mathbf u\left(t\right)\right) \nonumber \\
   & + \hat \delta \left(\boldsymbol{\nabla} C_2\left(t\right)\Delta C\left(t\right),\mathbf u\left(t\right)\right)=0  \quad a.e.\ on\ \left(0,T\right) \label{eq: uu3}.
\end{align}
For the term \eqref{eq: uu3}$_3$, we obtain the following estimate,
\begin{align}
   \left|\left(F\left(C_1\left(t\right)\right)\mathbf u\left(t\right),\mathbf u\left(t\right)\right)\right| 
   \le & \|F\left(C_1\left(t\right)\right)\|_{L^2}\|\mathbf u\left(t\right)\|^2_{L^4} \nonumber \\ 
   \le & M\|F\left(C_1\left(t\right)\right)\|_{L^2}\|\boldsymbol{\nabla}\mathbf u\left(t\right)\|_{L^2}\|\mathbf u\left(t\right)\|_{L^2} \nonumber \\ 
    \le& \epsilon \|\boldsymbol{\nabla}\mathbf u\left(t\right)\|^2_{L^2}+A\left(\epsilon\right)\|F\left(C_1\left(t\right)\right)\|^2_{L^2}\|\mathbf u\left(t\right)\|_{L^2}^2.
\end{align}
For the term \eqref{eq: uu3}$_4$, we obtain,
\begin{align}
   |(F\left(C_1\left(t\right)\right) & - F\left(C_2\left(t\right)\right),\mathbf u\left(t\right)\cdot \mathbf u_2  \left(t\right))| 
   \le \|F\left(C_1\left(t\right)\right)-F\left(C_2\left(t\right)\right)\|_{L^2}\|\mathbf u_2\left(t\right)\|_{L^4}\|\mathbf u\left(t\right)\|_{L^4} \nonumber \\
   & \le \|F\left(C_1\left(t\right)\right) -F\left(C_2\left(t\right)\right)\|_{L^2}\|\mathbf u_2\left(t\right)\|^\frac{1}{2}_{L^2}\|\boldsymbol{\nabla}\mathbf u_2\left(t\right)\|^\frac{1}{2}_{L^2}\|\mathbf u\left(t\right)\|^\frac{1}{2}_{L^2}\|\boldsymbol{\nabla}\mathbf u\left(t\right)\|^\frac{1}{2}_{L^2} \nonumber \\
   & \le \|F\left(C_1\left(t\right)\right) -F\left(C_2\left(t\right)\right)\|_{L^2}\|\boldsymbol{\nabla}\mathbf u_2\left(t\right)\|^\frac{1}{2}_{L^2}\|\mathbf u\left(t\right)\|^\frac{1}{2}_{L^2}\|\boldsymbol{\nabla}\mathbf u\left(t\right)\|^\frac{1}{2}_{L^2} \quad a.e. \ on\ \left(0,T\right) \nonumber \\
   & \le \|F\left(C_1\left(t\right)\right) -F\left(C_2\left(t\right)\right)\|^2_{L^2} + \|\boldsymbol{\nabla}\mathbf u_2\left(t\right)\|_{L^2}\|\mathbf u\left(t\right)\|_{L^2}\|\boldsymbol{\nabla}\mathbf u\left(t\right)\|_{L^2} \quad a.e. \ on\ \left(0,T\right) \nonumber \\
   & \le \|F\left(C_1\left(t\right)\right) -F\left(C_2\left(t\right)\right)\|^2_{L^2} + A\left(\epsilon\right)\|\boldsymbol{\nabla}\mathbf u_2\left(t\right)\|^2_{L^2}\|\mathbf u\left(t\right)\|^2_{L^2} + \epsilon\|\boldsymbol{\nabla}\mathbf u\left(t\right)\|^2_{L^2} \quad a.e. \ on\ \left(0,T\right).
\end{align}
For the term \eqref{eq: uu3}$_5$, we have,
\begin{align}
    \left|\left(\Delta C_1\left(t\right) \boldsymbol{\nabla} C\left(t\right) ,\mathbf{u}\left(t\right)\right)\right|
\leq& \|\Delta C_1\left(t\right)\|_{L^2} \|\boldsymbol{\nabla} C\left(t\right) \cdot \mathbf{u}\left(t\right)\|_{L^2}
\nonumber \\ \leq& M \|\Delta C_1\left(t\right)\|_{L^2} \|\boldsymbol{\nabla} C\left(t\right)\|_{L^2}^{1/2} \|\Delta C\left(t\right)\|_{L^2}^{1/2} \|\mathbf{u}\left(t\right)\|_{L^2}^{1/2} \|\boldsymbol{\nabla} \mathbf{u}\left(t\right)\|_{L^2}^{1/2}  \nonumber \\
\leq& 2\epsilon \left(\|\Delta C\left(t\right)\|_{L^2} \|\boldsymbol{\nabla} \mathbf{u}\left(t\right)\|_{L^2}\right) + A\left(\epsilon\right) \|\Delta C_1\left(t\right)\|_{L^2}^2 \|\boldsymbol{\nabla} C\left(t\right)\|_{L^2} \| \mathbf{u}\left(t\right)\|_{L^2} \nonumber \\
\leq& \epsilon \left( \|\boldsymbol{\nabla} \mathbf{u}\left(t\right)\|_{L^2}^2 + \|\Delta C\left(t\right)\|_{L^2}^2 \right) \nonumber \\&+ A\left(\epsilon\right) \|\Delta C_1\left(t\right)\|_{L^2}^2 \left( \|\boldsymbol{\nabla} C\left(t\right)\|_{L^2}^2  + \|\mathbf{u}\left(t\right)\|_{L^2}^2 \right) \quad a.e.\ on \ \left(0,T\right).
\end{align}
Using these inequalities, equation \eqref{eq: uu3} takes the form,
\begin{align}
    \frac{ 1}{2}\frac{d}{dt}\| \mathbf  u\left(t\right)\|_{L^2}+\mu _e\|\boldsymbol{\nabla}\mathbf{u}\left(t\right)\|^2_{L^2} 
    & \le \epsilon \|\boldsymbol{\nabla}\mathbf u\left(t\right)\|^2_{L^2}+A\left(\epsilon\right)\|F\left(C_1\left(t\right)\right)\|^2_{L^2}\|\mathbf u\left(t\right)\|_{L^2}^2 \nonumber \\ 
    & + \|F\left(C_1\left(t\right)\right)-F\left(C_2\left(t\right)\right)\|^2_{L^2} + A\left(\epsilon\right)\|\boldsymbol{\nabla}\mathbf u_2\left(t\right)\|^2_{L^2}\|\mathbf u\left(t\right)\|^2_{L^2} \nonumber \\ 
    & + \epsilon\|\boldsymbol{\nabla}\mathbf u\left(t\right)\|^2_{L^2}  + \epsilon \left( \|\boldsymbol{\nabla} \mathbf{u}\left(t\right)\|_{L^2}^2 + \|\Delta C\left(t\right)\|_{L^2}^2 \right) \nonumber \\ 
    & + A\left(\epsilon\right) \|\Delta C_1\left(t\right)\|_{L^2}^2 \left( \|\boldsymbol{\nabla} C\left(t\right)\|_{L^2}^2 + \|\mathbf{u}\left(t\right)\|_{L^2}^2 \right)  \quad a.e.\ on\ \left(0,T\right).
\end{align}
This implies,
\begin{align}
      \frac{ 1}{2}\frac{d}{dt}\| \mathbf  u\left(t\right)\|_{L^2}- &\epsilon \|\Delta C\left(t\right)\|_{L^2}^2+\left(\mu _e-3\epsilon\right) \|\boldsymbol{\nabla}\mathbf{u}\left(t\right)\|^2_{L^2}\le A\left(\epsilon\right)\|F\left(C_1\left(t\right)\right)\|^2_{L^2}\|\mathbf u\left(t\right)\|_{L^2}^2\nonumber\\& + 
    \|F\left(C_1\left(t\right)\right)-F\left(C_2\left(t\right)\right)\|^2_{L^2}+A\left(\epsilon\right)\|\boldsymbol{\nabla}\mathbf u_2\left(t\right)\|^2_{L^2}\|\mathbf u\left(t\right)\|^2_{L^2}\nonumber\\&+A\left(\epsilon\right) \|\Delta C_1\left(t\right)\|_{L^2}^2 \left( \|\boldsymbol{\nabla} C\left(t\right)\|_{L^2}^2 + \|\mathbf{u}\left(t\right)\|_{L^2}^2 \right)  \quad a.e.\ on\ \left(0,T\right).\label{eq: uu4}
\end{align}
Multiplying inequality \eqref{eq: unabC2} by $\hat \delta$ and subsequently adding it to inequalities \eqref{eq: uC3} and \eqref{eq: uu4}, we obtain the following:
\begin{align}
\frac{1}{2} \frac{d}{dt} \Big(\|C\left(t\right)\|^2_{L^2} 
       & + \hat \delta\|\boldsymbol{\nabla} C\left(t\right)\|^2_{L^2} + \|\mathbf u\left(t\right)\|_{L^2}^2 \Big) + 
     \left(d \hat \delta -3 \epsilon - 2\hat \delta\epsilon\right)\| \Delta C\left(t\right)\|^2_{L^2} \nonumber \\ 
     & + \left(d+\hat \delta \kappa-\hat \delta \epsilon\right)\|\boldsymbol{\nabla} C\left(t\right)\|^2_{L^2} 
     +\kappa \|C\left(t\right)\|^2_{L^2} + \left(\mu _e-3\epsilon\right) \|\boldsymbol{\nabla}\mathbf{u}\left(t\right)\|^2_{L^2} \nonumber \\ 
     & \le \left(M +\hat \delta A\left(\epsilon\right)\right)\|C\left(t\right)\|^2_{L^2}+A\left(\epsilon\right) \|\boldsymbol{\nabla}\mathbf u_1\left(t\right)\|^2_{L^2} \|\boldsymbol{\nabla} C\left(t\right)\|^2_{L^2} \nonumber\\ 
     & \quad + A\left(\epsilon\right) \|\boldsymbol{\nabla}\mathbf u\left(t\right)\|^2_{L^2} \|\mathbf u\left(t\right)\|^2_{L^2}  +\hat \delta  A\left(\epsilon\right) \|\boldsymbol{\nabla} C\left(t\right)\|^{2}_{L^2} \|\boldsymbol{\nabla} \mathbf u_1\|^{2}_{L^2} \nonumber \\ 
     & + A\left(\epsilon\right)\|F\left(C_1\left(t\right)\right)\|^2_{L^2}\|\mathbf u\left(t\right)\|_{L^2}^2 + \|F\left(C_1\left(t\right)\right)-F\left(C_2\left(t\right)\right)\|^2_{L^2} \nonumber \\ 
     & + A\left(\epsilon\right)\|\boldsymbol{\nabla}\mathbf u_2\left(t\right)\|^2_{L^2}\|\mathbf u\left(t\right)\|^2_{L^2} + A\left(\epsilon\right) \|\Delta C_1\left(t\right)\|_{L^2}^2 \left( \|\boldsymbol{\nabla} C\left(t\right)\|_{L^2}^2 + \|\mathbf{u}\left(t\right)\|_{L^2}^2 \right) \nonumber \\ 
     & \qquad\qquad\qquad\qquad\qquad\qquad\qquad\qquad\qquad\qquad\qquad  a.e.\ on\ \left(0,T\right). \label{eq: uuc1}
\end{align}
We choose $\epsilon$ such that $\left(d \hat \delta -3 \epsilon -2\hat \delta\epsilon\right)$, $\left(d \hat \delta -3 \epsilon -2\hat \delta\epsilon\right)$ and $\left(\mu _e-3\epsilon\right)$ are non negative.
Then from inequality \eqref{eq: uuc1} we get,
\begin{align}
\frac{1}{2} \frac{d}{dt} & \Big(\|C\left(t\right)\|^2_{L^2} + \hat \delta\|\boldsymbol{\nabla} C\left(t\right)\|^2_{L^2} +\|\mathbf u\left(t\right)\|_{L^2}^2 \Big) \nonumber \\ 
& \le \left(M+\hat\delta A\left(\epsilon\right)\right)\|C\left(t\right)\|^2_{L^2}+A\left(\epsilon\right) \|\boldsymbol{\nabla}\mathbf u_1\left(t\right)\|^2_{L^2} \|\boldsymbol{\nabla} C\left(t\right)\|^2_{L^2} \nonumber \\
& + A\left(\epsilon\right) \|\boldsymbol{\nabla}\mathbf u\left(t\right)\|^2_{L^2} \|\mathbf u\left(t\right)\|^2_{L^2} + \hat \delta A\left(\epsilon\right) \|\boldsymbol{\nabla} C\left(t\right)\|^{2}_{L^2} \|\boldsymbol{\nabla} \mathbf u_1\|^{2}_{L^2} \nonumber \\ 
& + A\left(\epsilon\right)\|F\left(C_1\left(t\right)\right)\|^2_{L^2}\|\mathbf u\left(t\right)\|_{L^2}^2 + \|F\left(C_1\left(t\right)\right)-F\left(C_2\left(t\right)\right)\|^2_{L^2} \nonumber \\ 
& + A\left(\epsilon\right)\|\boldsymbol{\nabla}\mathbf u_2\left(t\right)\|^2_{L^2}\|\mathbf u\left(t\right)\|^2_{L^2} + A\left(\epsilon\right) \|\Delta C_1\left(t\right)\|_{L^2}^2 \left( \|\boldsymbol{\nabla} C\left(t\right)\|_{L^2}^2 + \|\mathbf{u}\left(t\right)\|_{L^2}^2 \right) \quad a.e.\ on\ \left(0,T\right). \label{eq: uuc2} 
\end{align}
Integrating inequality \eqref{eq: uuc2} from $0$ to $\tau$ for any $\tau \in \left(0,T\right)$ we get,
\begin{align}
\int_0^\tau \frac{1}{2} \frac{d}{dt} & \Big(\|C\left(t\right)\|^2_{L^2} 
+ \hat \delta\|\boldsymbol{\nabla} C\left(t\right)\|^2_{L^2} +\|\mathbf u\left(t\right)\|_{L^2}^2  \Big) dt \nonumber \\ 
& \le \int_0^\tau\Big(\left(M + \hat \delta A\left(\epsilon\right)\right)\|C\left(t\right)\|^2_{L^2} + A\left(\epsilon\right) \|\boldsymbol{\nabla}\mathbf u_1\left(t\right)\|^2_{L^2} \|\boldsymbol{\nabla} C\left(t\right)\|^2_{L^2} \nonumber \\ 
& + A\left(\epsilon\right) \|\boldsymbol{\nabla}\mathbf u\left(t\right)\|^2_{L^2} \|\mathbf u\left(t\right)\|^2_{L^2} + \hat \delta  A\left(\epsilon\right) \|\boldsymbol{\nabla} C\left(t\right)\|^{2}_{L^2} \|\boldsymbol{\nabla} \mathbf u_1\|^{2}_{L^2} \nonumber \\ 
& + A\left(\epsilon\right)\|F\left(C_1\left(t\right)\right)\|^2_{L^2}\|\mathbf u\left(t\right)\|_{L^2}^2 + \|F\left(C_1\left(t\right)\right)-F\left(C_2\left(t\right)\right)\|^2_{L^2} \nonumber \\ 
& + A\left(\epsilon\right)\|\boldsymbol{\nabla}\mathbf u_2\left(t\right)\|^2_{L^2}\|\mathbf u\left(t\right)\|^2_{L^2} + A\left(\epsilon\right) \|\Delta C_1\left(t\right)\|_{L^2}^2 \left( \|\boldsymbol{\nabla} C\left(t\right)\|_{L^2}^2 + \|\mathbf{u}\left(t\right)\|_{L^2}^2 \right)\Big)dt. \label{eq: uuc3} 
\end{align}
Using assumption \eqref{assump:A1} we get,
\begin{align}
   \|F\left(C_1\right)-F\left(C_2\right)\|^2_{L^2\left(0,\tau;L^2\right)}\le M \|C\|^2_{L^2\left(0,\tau;L^2\right)} \quad \forall \tau \in \left(0,T\right). \label{eq: FMVT}
\end{align}
Then, by applying inequality \eqref{eq: FMVT} to inequality \eqref{eq: uuc3}, we obtain the following:
\begin{align}
\int_0^\tau \frac{1}{2} \frac{d}{dt} & \Big(\|C\left(t\right)\|^2_{L^2} + \hat \delta\|\boldsymbol{\nabla} C\left(t\right)\|^2_{L^2} +\|\mathbf u\left(t\right)\|_{L^2}^2 \Big) dt \nonumber \\ 
& \le  \left(2M+\hat \delta A\left(\epsilon\right) \right)\|C\|^2_{L^2\left(0,\tau;L^2\right)} + \int_0^\tau\Big(A\left(\epsilon\right) \|\boldsymbol{\nabla}\mathbf u_1\left(t\right)\|^2_{L^2} \|\boldsymbol{\nabla} C\left(t\right)\|^2_{L^2} \nonumber \\
& + A\left(\epsilon\right) \|\boldsymbol{\nabla}\mathbf u\left(t\right)\|^2_{L^2} \|\mathbf u\left(t\right)\|^2_{L^2} +\hat \delta  A\left(\epsilon\right) \|\boldsymbol{\nabla} C\left(t\right)\|^{2}_{L^2} \|\boldsymbol{\nabla} \mathbf u_1\|^{2}_{L^2} \nonumber \\
& + A\left(\epsilon\right)\|F\left(C_1\left(t\right)\right)\|^2_{L^2}\|\mathbf u\left(t\right)\|_{L^2}^2 + A\left(\epsilon\right)\|\boldsymbol{\nabla}\mathbf u_2\left(t\right)\|^2_{L^2}\|\mathbf u\left(t\right)\|^2_{L^2} \nonumber \\ 
& + A\left(\epsilon\right) \|\Delta C_1\left(t\right)\|_{L^2}^2 \left( \|\boldsymbol{\nabla} C\left(t\right)\|_{L^2}^2 + \|\mathbf{u}\left(t\right)\|_{L^2}^2 \right)\Big)dt. \label{eq: uuc4} 
\end{align}
Now, defining
\begin{align}
    \Phi\left(t\right) & = \frac{1}{2}\Big(2M+\hat \delta A\left(\epsilon\right)+A\left(\epsilon\right)\left(1+\hat \delta\right) \|\boldsymbol{\nabla}\mathbf u_1\left(t\right)\|^2_{L^2} \nonumber\\ 
    & + A\left(\epsilon\right) \|\Delta C_1\left(t\right)\|_{L^2}^2+A\left(\epsilon\right) \|\boldsymbol{\nabla}\mathbf u\left(t\right)\|^2_{L^2} + A\left(\epsilon\right)\left(1+\hat \delta\right) \|\boldsymbol{\nabla}\mathbf u_2\left(t\right)\|^2_{L^2}\Big),
\end{align}
and using \eqref{eq: uuc4}, we obtain
\begin{align}
    \int_0^\tau \frac{d}{dt} 
       \Big(\|C\left(t\right)\|^2_{L^2} + \hat \delta\|\boldsymbol{\nabla} C\left(t\right)\|^2_{L^2} +\|\mathbf u\left(t\right)\|_{L^2}^2  \Big) dt
   \nonumber\\ \le \int_0^\tau \Phi\left(t\right)\Big(\|C\left(t\right)\|^2_{L^2} + \hat \delta\|\boldsymbol{\nabla} C\left(t\right)\|^2_{L^2} +\|\mathbf u\left(t\right)\|_{L^2}^2  \Big) dt. \label{eq: uniqueness u+C}
\end{align}
This can be written as,
\begin{align}
  \int_0^\tau \frac{d}{dt} 
     \Big(  \exp{\Big(-\int_0^t \Phi\left(s\right)ds\Big)} \Big(\|C\left(t\right)\|^2_{L^2} + \hat \delta\|\boldsymbol{\nabla} C\left(t\right)\|^2_{L^2} +\|\mathbf u\left(t\right)\|_{L^2}^2  \Big) \Big) dt \le 0.
\end{align}
This implies,
\begin{align}
     \exp{\Big(-\int_0^\tau \Phi\left(s\right)dt\Big)} \Big(\|C\left(\tau\right)\|^2_{L^2} + \hat \delta\|\boldsymbol{\nabla} C\left(\tau\right)\|^2_{L^2} +\|\mathbf u\left(\tau\right)\|_{L^2}^2  \Big) \nonumber\\ \le \Big(\|C\left(0\right)\|^2_{L^2} + \hat \delta\|\boldsymbol{\nabla} C\left(0\right)\|^2_{L^2} +\|\mathbf u\left(0\right)\|_{L^2}^2  \Big) \quad \forall \tau \in \left(0,T\right).
\end{align}
Now $\mathbf u\left(0\right)=0$ and $C\left(0\right)=0$ ensures the uniqueness of the solution $\mathbf u$ and $C$.

The uniqueness of the pressure again follows from de Rham's theorem (\emph{\citep[Theorem IV.2.4]{boyer2012mathematical}}). Thus, the weak solution to the system is unique.

\begin{corollary}
    In the absence of a chemical reaction ($\kappa = 0$), equation \eqref{uniform boundedness of C_n 1st} in Lemma \ref{lemma:L1} takes the following form:
\begin{align}
    \frac{1}{2}\frac{d}{dt}\|C_n\left(t\right)\|_{L^2}^2+d\|\boldsymbol{\nabla}{C_n\left(t\right)}\|^2_{L^2}=0 \quad a.e. \ on \ (0,T). \label{uniform boundedness of C_n no reaction}
\end{align}
Integrating \eqref{uniform boundedness of C_n no reaction} from $0$ to $T$ gives us 
$C_n$ is uniformly bounded on $L^\infty(0,T;L^2(\Omega))\cap L^2(0,T;H^1(\Omega))$ for any $T>0$. Hence, for $\kappa=0$ we get the existence of the solution for any $T>0$. 
\end{corollary}

\section{Some special cases} \label{sec:special_cases}

In this section, we discuss some particular $F(C)$ that are relevant to specific cases of carbon sequestration and/or oil recovery. Although the polynomial forms of $F(C)$ are not of direct practical interest, we include them here to complete the mathematical analysis of the model. 

\begin{corollary}
\textbf{Constant Function\ (\( F\left(C\right) = a \), with \( a \geq 0 \)):}
\end{corollary}
\begin{proof}
       In this case, \( F \) is a mapping \( F : L^2\left(0,T;\mathscr V_2\right) \to L^2\left(0,T;H^1\right) \). It is straightforward to show that Theorems \ref{th: existence}--\ref{th: Uniqueness} hold.
\end{proof}

    \begin{corollary}
    \textbf{Polynomial Function (\( F\left(C\right) = a_0 + a_1C + \dots + a_lC^l \), where \( a_i \geq 0 \) for \( 0 \leq i \leq l \)) with $C_0>0$:}
    \end{corollary}
    \begin{proof}
    
\textbf{Two dimensional case:}
    For a polynomial \( F\left(C\right) \), we establish the following:
    \begin{itemize}
        \item Lemmas \ref{lemma:L1} and \ref{lemma:L3} can be verified.  
        \item Proceeding similarly as the proof of positiveness of solute concentration  
        we can say that\[
C_n \ge 0 \quad a.e. \ on \ (0,T) \times \Omega.
\] This implies $\|\sqrt{ F\left(C_n\left(t\right)\right) }\mathbf u_n\left(t\right) \|_{L^2}^2$ is well defined almost everywhere on  $\left(0,T\right)$, which indicates the validity of Lemma \ref{lemma:L2}.
        \item To prove Lemma \ref{lemma:L4}, it suffices to show that 
        \(
        \|\boldsymbol{\nabla} F\left(C\left(t\right)\right)\|_{L^2\left(0,T;L^2\left(\Omega\right)\right)} < \infty.
        \)
        \begin{align}
            \|\boldsymbol{\nabla} F\left(C\left(t\right)\right)\|_{L^2} &\leq \|\left(a_1C\left(t\right) + 2a_2C\left(t\right)^2 + \dots + la_lC\left(t\right)^{l-1}\right) \boldsymbol{\nabla} C\left(t\right)\|_{L^2} \nonumber\\
            &\leq \|\left(a_1C\left(t\right) + 2a_2C\left(t\right)^2 + \dots + la_lC\left(t\right)^{l-1}\right)\|_{L^4} \|\boldsymbol{\nabla} C\left(t\right)\|_{L^4}.
        \end{align}
        Using Gagliardo–Nirenberg inequality (Theorem \ref{Th: gagliardo}), Rellich-Kondrasov theorem (\emph{\citep[Theorem 2.6.3]{kesavan2015topics}}), Lemma \ref{lemma:L2}, and Young' inequality we obtain,
     \begin{align}
        \|\boldsymbol{\nabla} F\left(C\left(t\right)\right)\|_{L^2} \leq M\left(1+ \|\Delta C\left(t\right)\|_{L^2}\right) \quad a.e. \ on \ \left(0,T\right).
         \end{align}
        Thus, \( \|\boldsymbol{\nabla} F\|_{L^2\left(0,T;L^2\left(\Omega\right)\right)} < \infty \), and Lemma \ref{lemma:L4} is satisfied.
        
        \item Furthermore, using a similar approach as in the above inequality, it can be shown that
        \[
        \|F\left(C_n\left(t\right)\right) - F\left(C\left(t\right)\right)\|_{L^2} \to 0 \quad a.e.\ on \ \left(0,T\right).
        \]
        This establishes the validity of \eqref{eq: Fu}.  
        \item Additionally, a Lipschitz-like condition,
        \[
        \|F\left(C_1\right) - F\left(C_2\right)\|^2_{L^2\left(0,\tau;L^2(\Omega)\right)} \leq M \|C_1 - C_2\|^2_{L^2\left(0,\tau;H^1(\Omega)\right)}
        \]
        is satisfied.
   \end{itemize}
Hence, Theorems \ref{th: existence}--\ref{th: Uniqueness} hold in this case as well.
    
\textbf{Three-dimensional case:} In three dimensions, the Theorems \ref{th: existence} and \ref{th: bddness & blow_up} remain valid provided that either of the following conditions holds:
    \begin{enumerate}
        \item The initial concentration satisfies $0\le C_0\le 1$ almost everywhere on $\Omega$.
        \item The domain $\Omega$ has a boundary of class $\mathcal{C}^2$.
    \end{enumerate}
\begin{enumerate}
    \item If  $0\le C_0\le 1$ almost everywhere in $\Omega$, then proceeding similarly as the proof of Theorem \ref{th: bddness & blow_up} we say that
\begin{align}
0\leq C_n\leq1\quad a.e. \ on\ (0,T)\times \Omega. \label{ineq: 3d pol 0<c<1}
\end{align}
Using this estimate in equation \eqref{L2u1} we obtain the following inequality,
  \begin{align}
      \frac{1}{2} \frac{d}{dt} \left\| \mathbf{u}_n\left(t\right) \right\|_{L^2}^2 +\mu_e\left(\boldsymbol{\nabla} \mathbf{u}_n\left(t\right),\boldsymbol{\nabla} \mathbf{u}_n\left(t\right)\right) + a_0\left(
\mathbf{u}_n\left(t\right),\mathbf{u}_n\left(t\right)\right) = \left \langle \boldsymbol{\nabla} \cdot \mathbf{T}\left(C_n\left(t\right)\right),\mathbf{u}_n\left(t\right) \right \rangle \nonumber\\+\left(\mathbf f\left(t\right),\mathbf{u}_n\left(t\right)\right) \quad a.e. \ on\ \left(0,T\right).
        \end{align}
The rest of the proof of Lemma \ref{lemma:L2} follows as before. 
 To prove Lemma \ref{lemma:L4}, we need to bound the term $\left|\left(F(C_n(t))\mathbf u_n(t), \mathbf w^1\right)\right|$, which can be achieved using the upper bound of $C_n.$ The rest of the proof of Theorem \ref{th: existence} follows as before.
 \item If $\Omega$ has a boundary of class $\mathcal C^2$ then we have the embedding \begin{align}
     H^2(\Omega) \hookrightarrow L^\infty(\Omega).
 \end{align}
 Thus we have $C\in L^2(0,T;L^\infty(\Omega)).$
    \end{enumerate}
    \end{proof} 

    \begin{corollary}
    \textbf{Exponential Function (\( F\left(C\right) = \exp\left(RC\right) \), with \( R \in \mathds{R} \)):} 
    \end{corollary}
    \begin{proof} 

\textbf{Two dimensional case:}
    For \( F\left(C\right) = \exp\left(RC\right) \), the following verifications are made:
    \begin{itemize}
        \item Lemmas \ref{lemma:L1}, \ref{lemma:L2}, and \ref{lemma:L3} are satisfied.
        
        \item To prove Lemma \ref{lemma:L4}, we establish that \( \|\exp\left(RC_n\right)\|_{L^2\left(0,T;H^1\left(\Omega\right)\right)} \) is finite. We have the following estimate,
        \begin{align}
        \|\exp\left(RC_n\left(t\right)\right)\|_{L^4}^4 &\leq \sum_{k=0}^\infty \int_\Omega \frac{\left(R^{4k} C_n^{4k}\left(\mathbf{x},t\right)\right)}{k!} \, d\Omega \nonumber \\
        &\leq \sum_{k=0}^\infty \frac{R^{4k}}{k!} \|C_n\left(t\right)\|_{L^k}^{4k}.
        \end{align}
        Using Lemma \ref{lemma:L2} and Rellich-Kondrasov theorem (\emph{\citep[Theorem 2.6.3]{kesavan2015topics}}), it follows that,
        \begin{align}
        \|\exp\left(RC_n\right)\|_{L^\infty\left(0,T;L^4\left(\Omega\right)\right)} < \infty \quad \forall n \in \mathds{N}.
        \end{align}
        Thus, we have,
        \begin{align}
            \|\boldsymbol{\nabla} F\left(C_n\left(t\right)\right)\|_{L^2} &\leq |R| \|\exp\left(RC_n\left(t\right)\right)\|_{L^4} \|\boldsymbol{\nabla} C_n\left(t\right)\|_{L^4}.
        \end{align}
        This implies \( F\left(C_n\right) \) is uniformly bounded in \( L^2\left(0,T;H^1\right) \), and Lemma \ref{lemma:L4} holds.

       \item Now using Lemma \ref{lemma:L2} and Rellich-Kondrasov theorem (\emph{\citep[Theorem 2.6.3]{kesavan2015topics}}) it can be shown that
       \begin{align}
 \|F\left(C_n\left(t\right)\right)-F\left(C\left(t\right)\right)\|^2_{L^2}\le& \|C_n\left(t\right)-C\left(t\right)\|^2_{L^4}  \sum_{k=0}^\infty \frac{1}{k!} \sum_{l=0}^{k-1} \|C_n^l\left(t\right)C^{k-1-l}\left(t\right)\|^2_{L^4} \nonumber\\
 \le& M \|C_n\left(t\right)-C\left(t\right)\|^2_{H^1} \quad a.e.\  on \ \left(0,T\right) \label{eq: exp limit pass},
 \end{align}
        which validates \eqref{eq: Fu}.
        
        \item Using similar techniques as of inequality \eqref{eq: exp limit pass} we can show,
 \begin{align}
     \|F\left(C_1\right) - F\left(C_2\right)\|^2_{L^2\left(0,\tau;L^2(\Omega)\right)} \le M \|C_1 - C_2\|^2_{L^2\left(0,\tau;H^1(\Omega)\right)}.
 \end{align} 
    Hence, Theorems \ref{th: existence}--\ref{th: Uniqueness} hold for the exponential case as well. 
\end{itemize}

For the three-dimensional case, the same additional assumptions as in the polynomial case are required, and the proof follows similarly.
\end{proof} 

\begin{remark}
   Thanks to Rellich-Kondrasov theorem (\emph{\citep[Theorem 2.6.3]{kesavan2015topics}}), we have the embedding $H^1(\Omega)\hookrightarrow L^q(\Omega)$ for $1\leq q <\infty$, which helps to prove polynomial and exponential cases for a two-dimensional domain. In contrast, in three dimensions, this result does not hold. For which we require an additional assumption on $C_0$ or a smoother domain. 
\end{remark}

\section{Numerical validation} \label{sec:numerical}

The initial-boundary value problem \eqref{eq:continuity}--\eqref{eq:initial_condition} is numerically solved using the finite element solver of COMSOL \textsc{Multiphysics} in a circular domain $\Omega (\subset \mathds R^2)$ of radius $0.5$ centered at the origin. Under suitable initial conditions, the non-negativity of the solute concentration, finite-time blow-up, and long-time decay to zero are discussed below. The weak formulation \eqref{var u}--\eqref{var C} are used for numerical computation. The computational domain is discretised using an unstructured triangular mesh. A grid independence study is performed for different element sizes to show that the solution is not significantly affected by spatial discretisation. For this purpose, a sequence of successively refined  meshes of element sizes 
\begin{align*}
h=2^{-4},2^{-5},2^{-6},2^{-7},2^{-8},2^{-9}   
\end{align*} 
have been considered. The convergence of the solution is examined by computing the relative absolute error between successive mesh refinements, defined as
\begin{align}
    \| e_h \|_{L^1} = \frac{\int_0^T\int_\Omega\left| C_{h}(\mathbf x,t) - C_{h/2}(\mathbf x,t)\right|d\mathbf xdt}{\int_0^T\int_\Omega\left|C_{h/2}(\mathbf x,t)\right|d\mathbf xdt} 
\end{align}

In particular, we choose 
\begin{align}
\label{eq:numerical_IC1}
 \mathbf u_0 = \mathbf 0, \qquad \text{and} \qquad   C_0 = \begin{cases}
       0.9, & x^2 + y^2 \leq 0.04, \\
       0, & \text{otherwise.}     \end{cases} 
\end{align}
The parameters are taken as $\mu_e = 0.1, d = 10^{-9}, \hat{\delta} = 10^{-8}, \kappa = 0.001, \mu = e^{3C} \times 10^{-3}, K = 10^{-10}$ to validate the long-time decay of the solute concentration as $t \rightarrow \infty$. 
The grid convergence study for this initial condition has been presented in Table \ref{tab:grid_independence}. 
\begin{table}[h!]
\centering
\caption{Grid independence study showing relative errors between successive meshes.}
\label{tab:grid_independence}
\begin{tabular}{cccccc}
\hline
$h$  
& $2^{-4}$ 
& $2^{-5}$
& $2^{-6}$ 
& $2^{-7}$
& $2^{-8}$
\\
\hline
$\|e_h\|_{L^1}$ 
& $2.1 \times 10^{-1}$ 
& $1.1 \times 10^{-1}$ 
& $4.4 \times 10^{-2}$ 
& $4.6 \times 10^{-2}$
& $1.8 \times 10^{-2}$
\\
\hline
\end{tabular}
\end{table}
It can be observed that the relative error decreases as the mesh is refined. In particular, the relative error between the two finest meshes ($2^{-8}$ and $2^{-9}$) is approximately $1.8\%$. This indicates that the numerical solution has reached a grid-independent state. Therefore, we take the mesh corresponding to the element size $h  = 2^{-8}$ for the subsequent simulations. 

\begin{figure}[h!]
     \centering
    \includegraphics[width=0.33\textwidth]{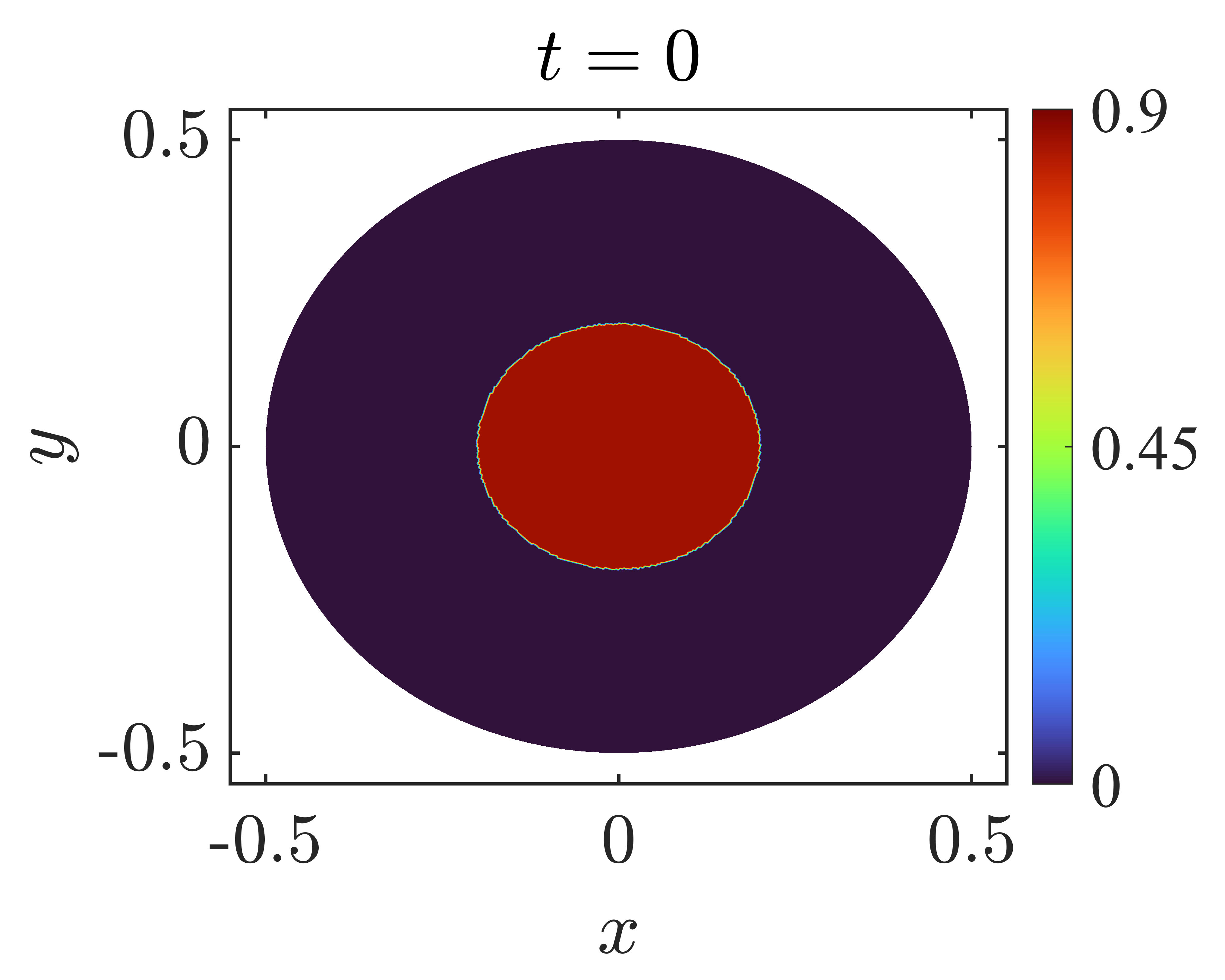}
     \includegraphics[width=0.33\textwidth]{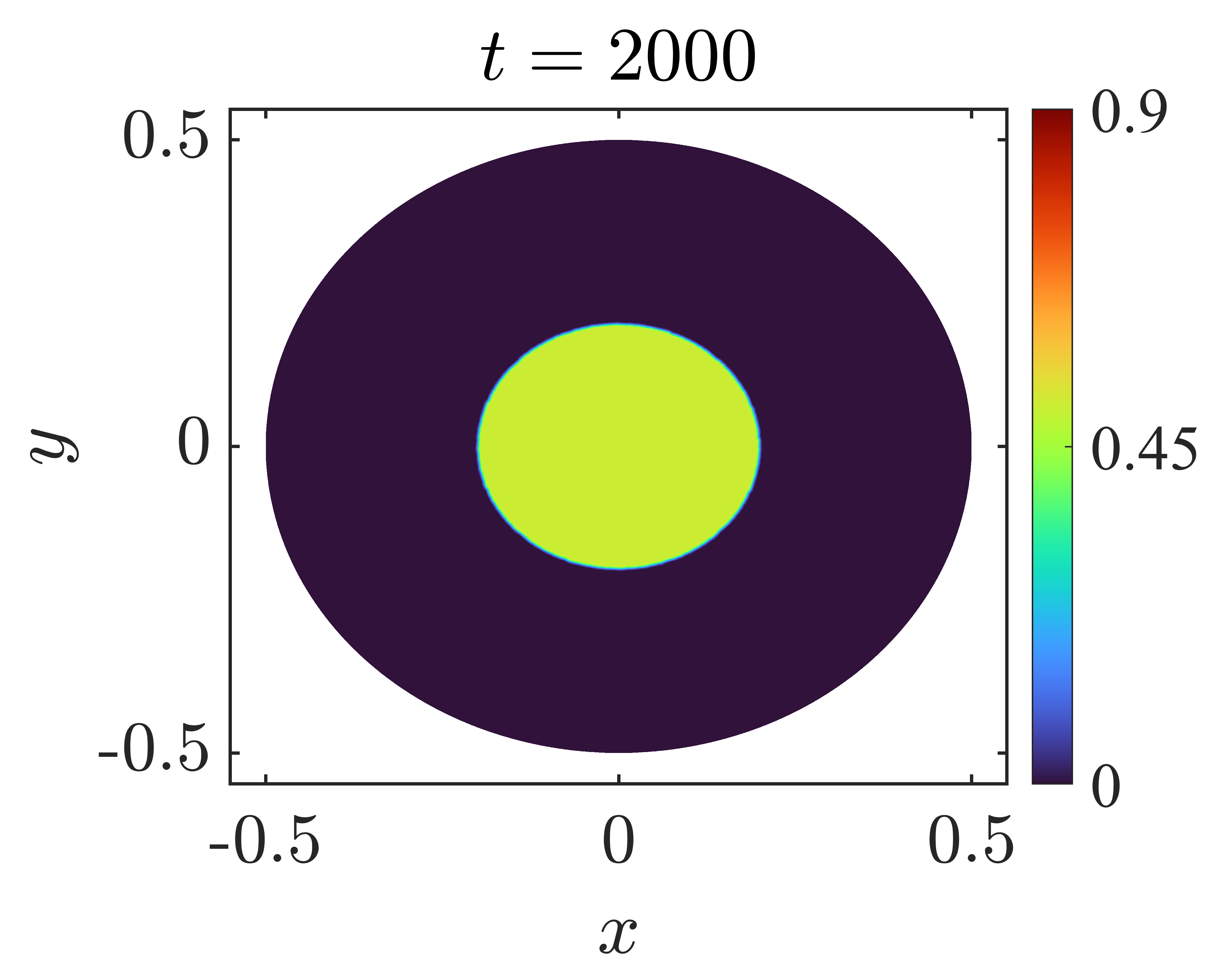}  
     \includegraphics[width=0.33\textwidth]{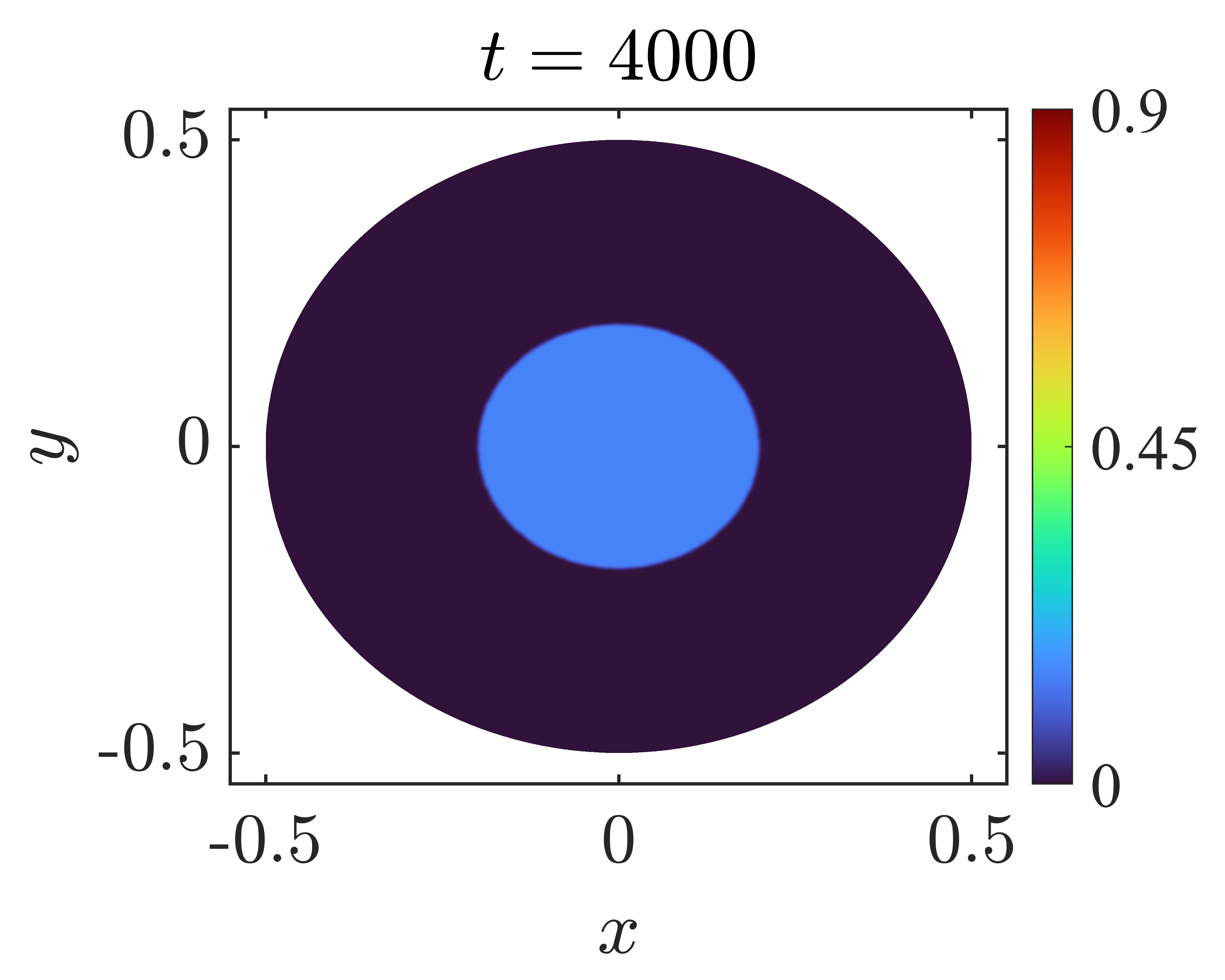} \\
    \includegraphics[width=0.33\textwidth]{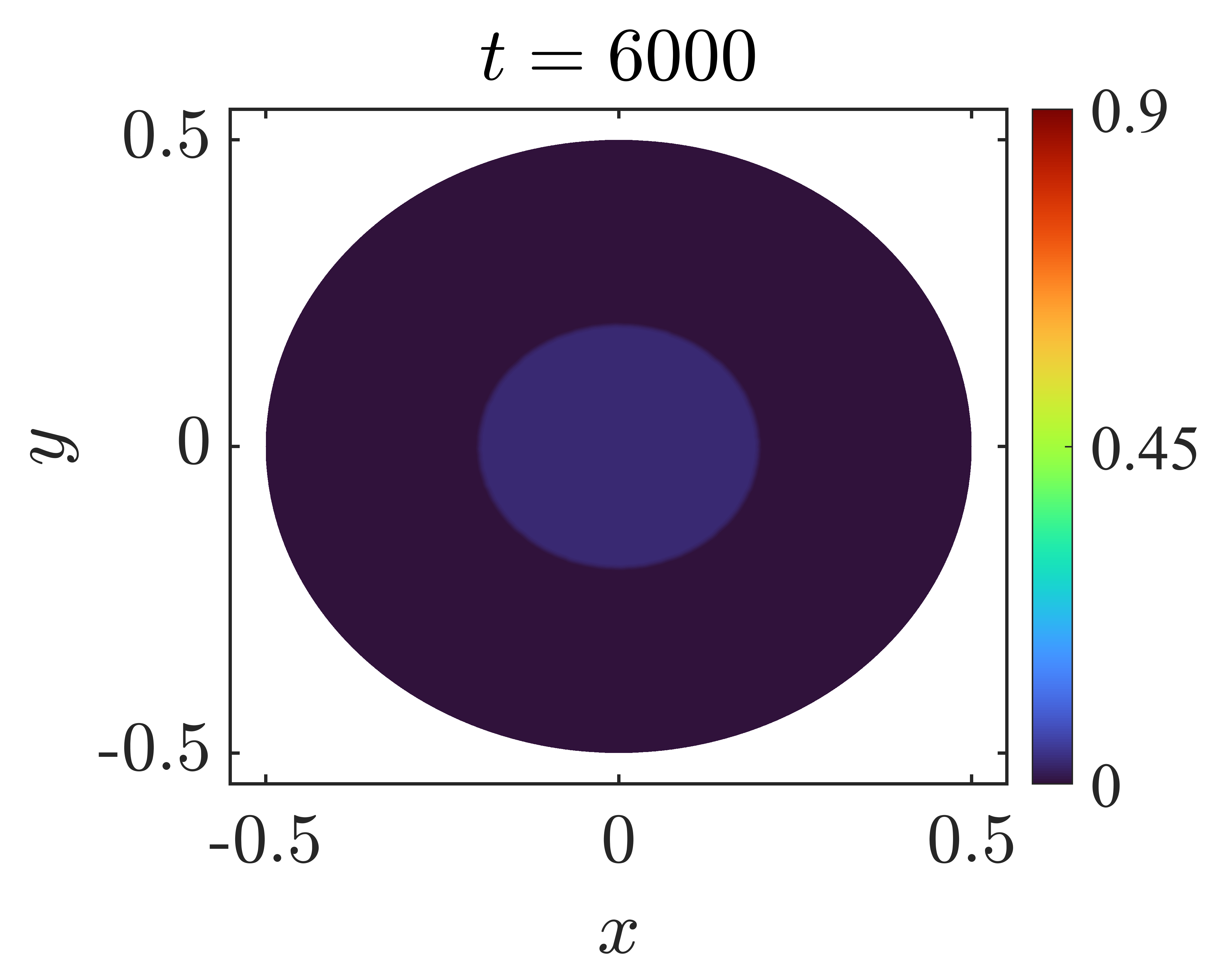}  
     \includegraphics[width=0.33\textwidth]{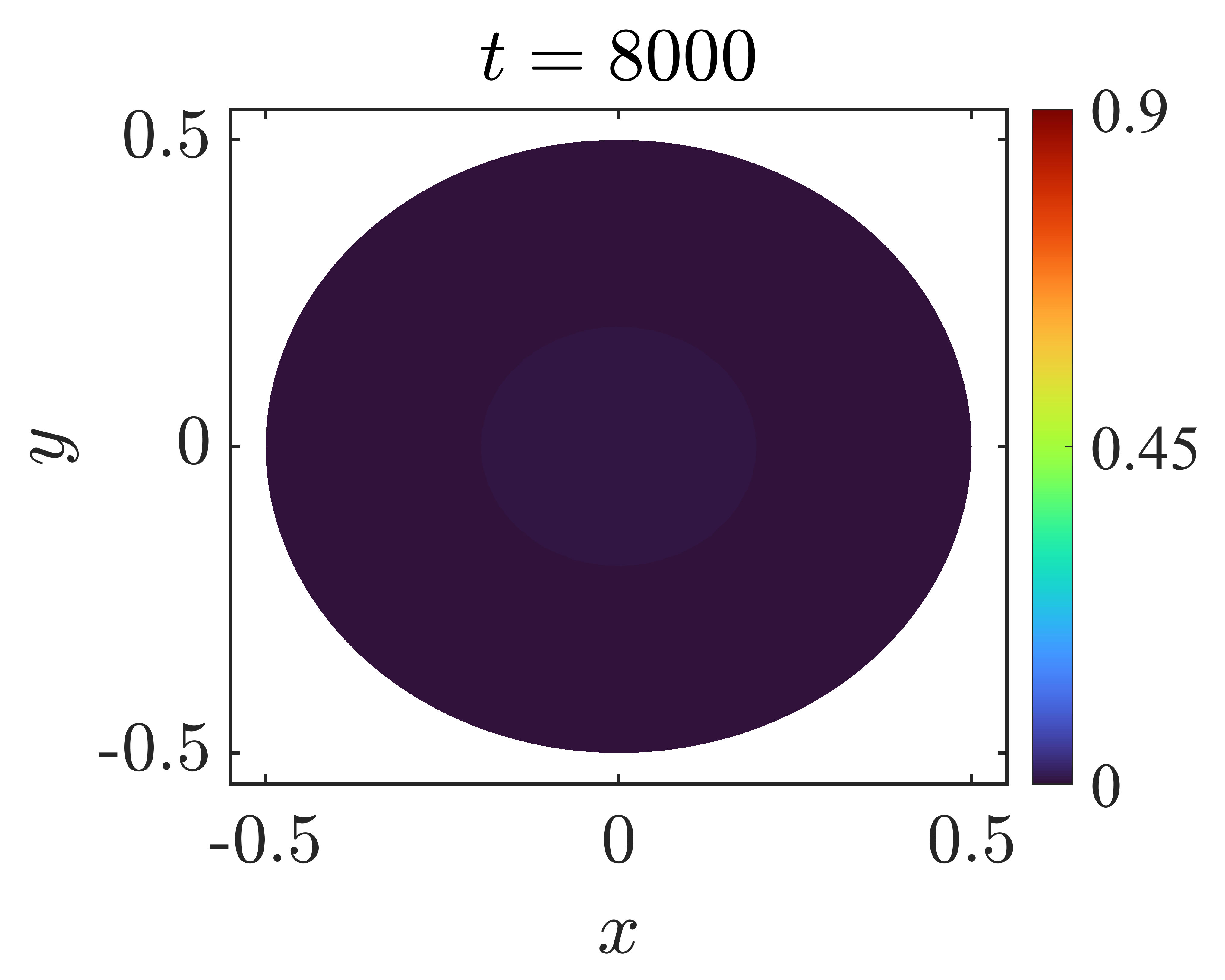}  
     \includegraphics[width=0.33\textwidth]{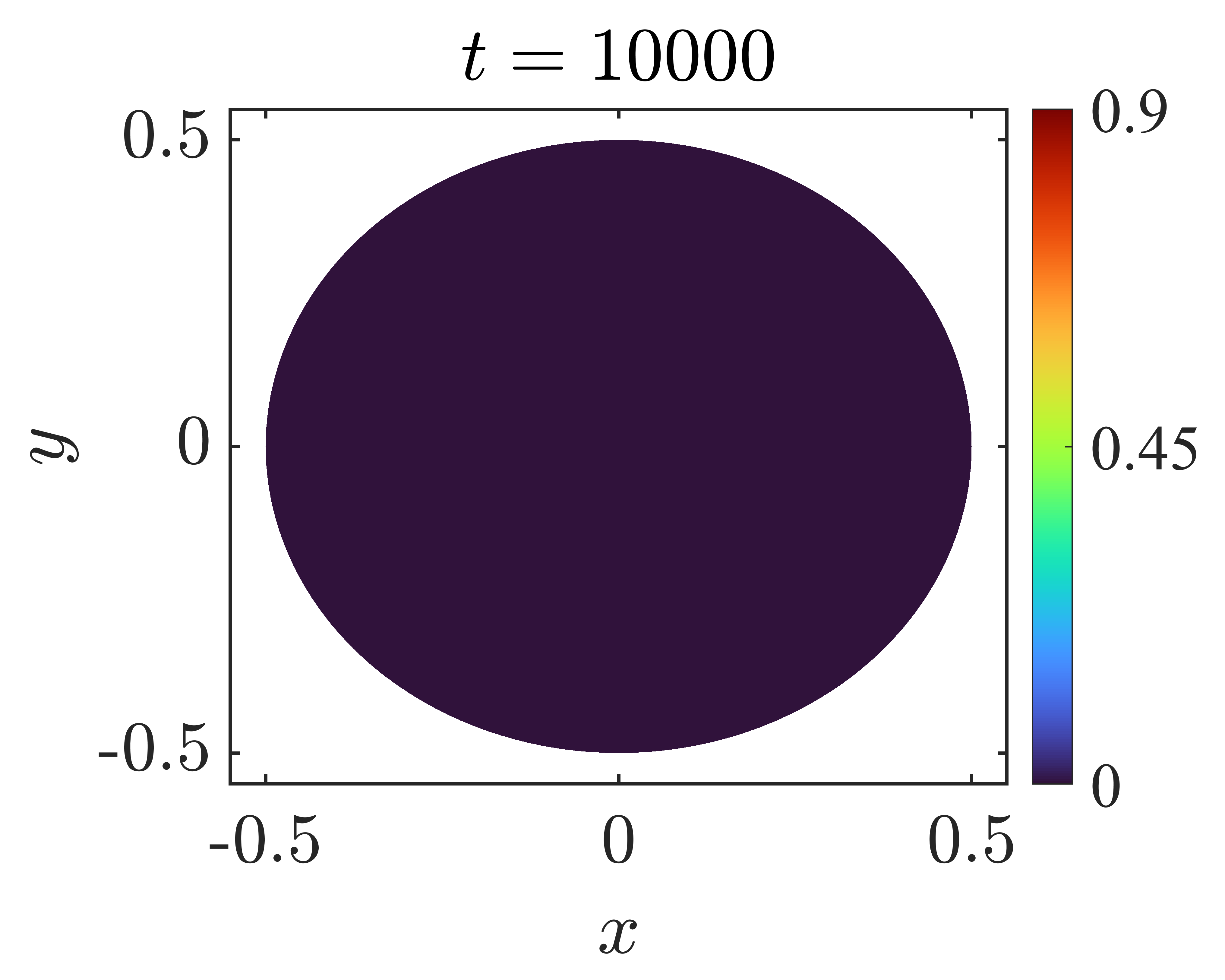}   
\caption{Contour plots of solute concentration $C$ for different time levels corresponding to the initial condition \eqref{eq:numerical_IC1}.}
\label{fig: contour}
\end{figure}

Figure \ref{fig: contour} depicts contour plots of the solute concentration at different time levels, showing non-negativity and decay to zero as time increases. Figure \ref{fig:long-time}(a) shows the exponential decay of the total mass over time, which agrees with the theoretical estimate in inequality \eqref{uniform boundedness C for all time 2nd}.  

\begin{figure}[h!]
\centering
\includegraphics[width=0.49\textwidth]{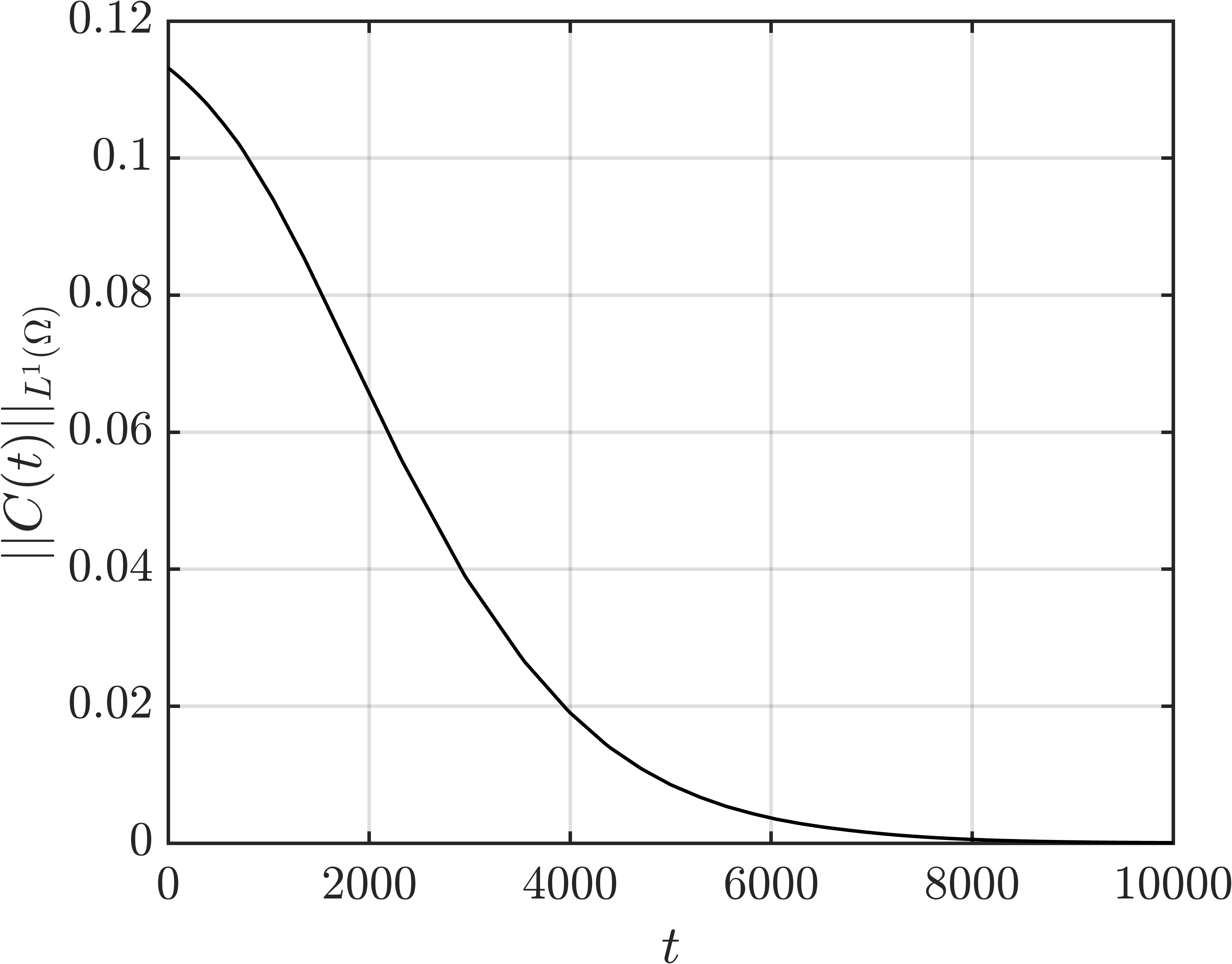}
\includegraphics[width=0.49\textwidth]{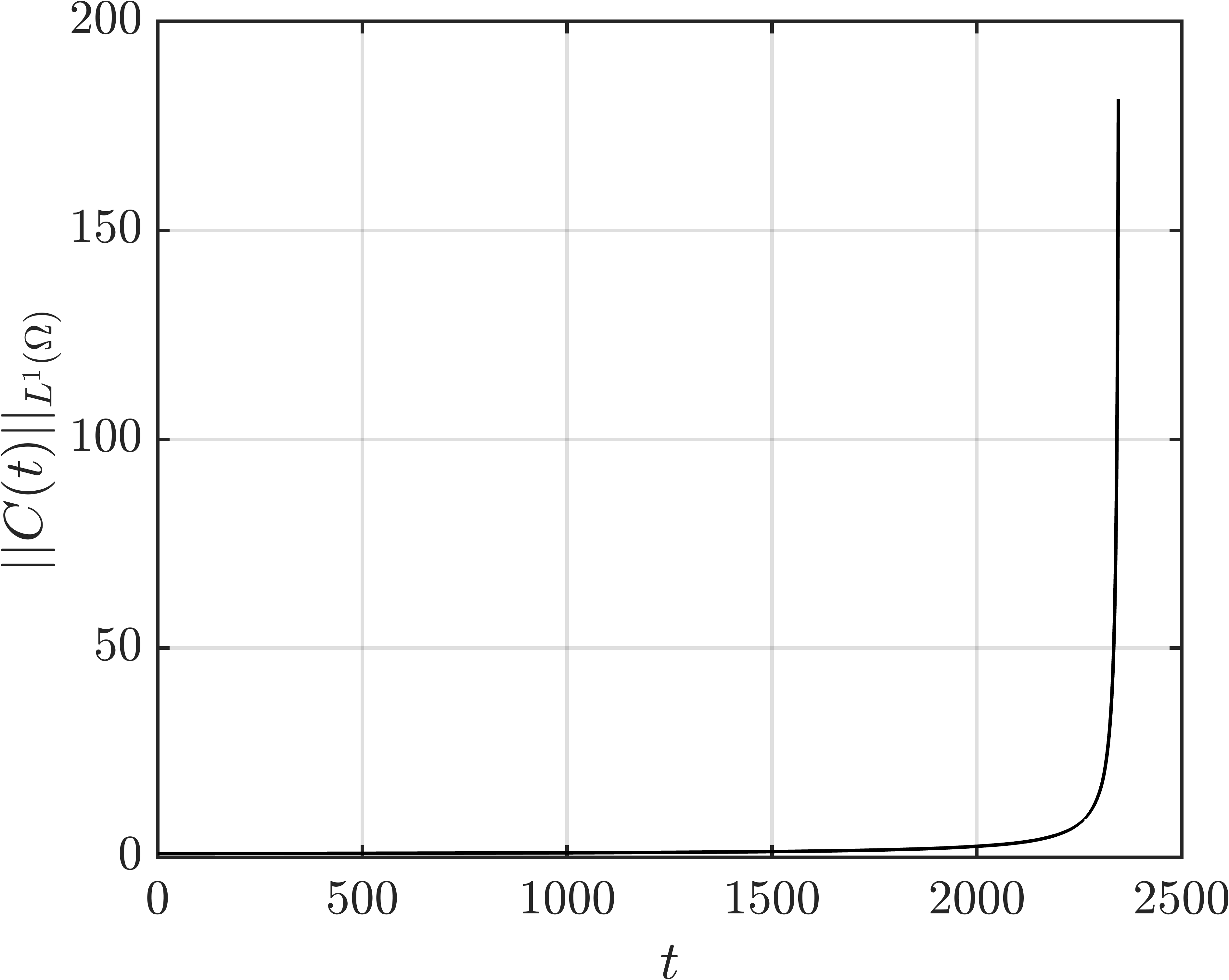}
\caption{Temporal evolution of $\|C(t)\|_{L^1(\Omega)}$ corresponding to the initial conditions given in (a) equation \eqref{eq:numerical_IC1}, and (b) equation \eqref{eq:numerical_IC2}.}
\label{fig:long-time}
\end{figure}

Next, we discuss the finite-time blow-up solution of the model problem. In this regard, we use the following initial condition 
\begin{align}
\label{eq:numerical_IC2}
 \mathbf u_0 = \mathbf 0, \qquad \text{and} \qquad  C_0 = 1.1.  
\end{align}
The parameter values are kept unchanged from those in the previous configuration. In an auto-catalytic reaction, the process starts slowly, but as time proceeds, more reactants are produced, which further accelerates the overall production rate. Figure \ref{fig:long-time}(b) shows that the solute concentration blows up in finite time. Using these parameter values, the theoretical estimate for the blow-up time obtained from equation \eqref{eq: blowup time} yields $T_\ast \approx 2397.895$. Numerically, the blow-up time is computed to be $2345.71$, in good agreement with the theoretical estimate. 

Overall, the numerical simulations validate the theoretical findings of the non-negativity, decay and blow-up behaviour based on the initial condition for the concentration. 

\section{Conclusions \label{sec: conclusion}}

 In this paper, we proved the well-posedness of a model that uses an unsteady Darcy-Brinkman equation with Korteweg stress coupled with an advection-diffusion-reaction equation describing miscible displacement in a porous medium. The viscosity of the fluid and the permeability of the porous medium are assumed to depend on the concentration of a solute, which is also responsible for the Korteweg stress. We established the existence of a local in time weak solution of the said model in a bounded domain with a smooth boundary. We showed that if the initial concentration $C_0$ is bounded between $0$ and $1$, then the solution is global in time, and if $C_0>1$, then the solute concentration $C$ blows up in a finite time. We established the uniqueness of the solution in a bounded domain with a smooth boundary. We generalised our analysis in terms of the mobility function $F(C)$, which represents the ratio of fluid viscosity to the medium's permeability. Furthermore, we elaborated our results for some particular functional forms of the viscosity and permeability relevant in the contexts of enhanced oil recovery \citep{homsy1987viscous}, geological carbon sequestration \citep{huppert2014fluid}, and chromatography separation \citep{rana2019influence}. 
Our analysis provides the framework for understanding a wide class of problems describing miscible flows in porous media. In particular, the effects of chemical reactions \citep{nagatsu2014hydrodynamic}, fluid compressibility \citep{morrow2023gas}, and the Darcy/Brinkman-Forchheimer model for miscible displacement of incompressible fluids in porous media \citep{dehghan2024fully, caucao2020fully, caucao2022posteriori} are the focus of our ongoing research. 

\section*{Acknowledgements}
P.R. acknowledges financial support received from the University of
Grants Commission, Government of India (File No:  211610146712), and research facilities provided by the Indian Institute of Technology Guwahati, India. S.P. acknowledges financial support through the Start-Up Research Grant (SRG/2021/001269), MATRICS Grant (MTR/2022/000493) from the Science and Engineering Research Board, Department of Science and Technology, Government of India, and Start-up Research Grant (MATHSUGIITG01371SATP002), IIT Guwahati. 

\bibliographystyle{abbrvnat}
\bibliography{ref}

\end{document}